\documentclass[12pt]{amsart}
\usepackage{amsthm,amssymb,amsmath,amscd,epic,eepic}
\usepackage[all]{xy}
\usepackage{xcolor}
\usepackage{bm}

\usepackage{enumitem}
\usepackage[normalem]{ulem}
\usepackage{cancel}
\usepackage{stmaryrd}
\usepackage{comment}

\usepackage{relsize}
\newcommand     {\printname}[1] {}

\newcommand{\mute}[2][1]{}

\newcommand{\sm}{\left(M,\omega\right)}

\newcommand{\rorbit}{\mathcal{O}} 
\newcommand{\torbit}{\mathfrak{O}} 

\numberwithin{equation}{section}
\newtheorem {Theorem}[equation]                   {Theorem}
\newtheorem*{Theorem*}                   {Theorem}

\newtheorem {Lemma}[equation]           {Lemma}

\newtheorem {Corollary} [equation]      {Corollary}
\newtheorem* {Corollary*}                {Corollary}
\newtheorem {Proposition} [equation]    {Proposition}
\newtheorem*{Proposition*}{Proposition}

\newtheorem*{Question*}{Question}

\newtheorem* {Lemma*}                    {Lemma}

\theoremstyle{definition}
\newtheorem{Definition}[equation]{Definition}
\newtheorem*{Definition*}{Definition}

\theoremstyle{remark}
\newtheorem{Remark}[equation]{Remark}
\newtheorem{Observation}[equation]{Observation}
\newtheorem*{Remark*}{Remark}

\newtheorem{Example}[equation]{Example}

\long\def\symbolfootnote[#1]#2{\begingroup%
\def\thefootnote{\fnsymbol{footnote}}\footnote[#1]{#2}\endgroup}

\def \Re {{\mathfrak R}}
\def \Im {{\mathfrak I}}

\def \del	{{\partial}}

\def    \inv    {^{-1}}

\def \MmodT {{ M/\!/T}}
\def    \YmodT {{ Y/\!/ T}}

\def    \CP	{{\mathbb C}{\mathbb P}}
\def    \R	{{\mathbb R}}
\def    \N	{{\mathbb N}}

\def    \C	{{\mathbb C}}
\def    \Z       {{\mathbb Z}}

\def    \ft	{{\mathfrak t}}
\def    \fh	{{\mathfrak h}}

\def	\ker 	{{\operatorname{ker}}}
\def	\exp	{{\operatorname{exp}}}

\begin{document}



\title[Connectedness beyond semitoric systems II]{Connectedness of fibers beyond semitoric systems II: ephemeral singular points}

\author[Daniele Sepe]{Daniele Sepe}
\address{Escuela de Matem\'aticas, Universidad Nacional de Colombia sede Medell\'in, Colombia.}
\email{dsepe@unal.edu.co}

\author[Susan Tolman]{Susan Tolman}
\address{Department of Mathematics, University of Illinois at Urbana-Champaign,
 Urbana, IL 61801.}
\email{tolman@illinois.edu}


\begin{abstract}
  In an earlier paper, we proved the connectedness of the fibers of every $2n$-dimensional integrable system satisfying both:
\begin{enumerate}[label=(\Alph*),ref=(\Alph*)]
 \item \label{item:extend-abs} the action extends the action of an $(n-1)$-dimensional torus which has a proper moment map, and
 \item \label{item:nondeg-abs} every tall singular point is non-degenerate and no such point has a hyperbolic block and connected $T$-stabilizer.   
   \end{enumerate}
Unfortunately, these criteria are fairly restrictive.
 Our main goal in this paper is to find a larger class of integrable systems that has connected fibers by weakening the non-degeneracy assumption above. To achieve this, we introduce ``ephemeral" degenerate singular points, examples of which have appeared in the literature in the context of both $p \! : \! -q$ resonances  and special Lagrangian fibrations. Finally, we construct a family of examples that shows that our main theorem meaningfully extends previous results.
\end{abstract}

\maketitle

\tableofcontents

\section{Introduction}

Let $T$ be an $(n-1)$-dimensional torus with Lie algebra $\ft$. Building on \cite{At,GS2,LMTW,VN,Wa}, in \cite[Theorem 1.6]{ST1} we proved the connectedness of the fibers of every $2n$-dimensional integrable system such that:
\begin{enumerate}[label=(\Alph*),ref=(\Alph*)]
 \item \label{item:extend} the action extends a complexity one $T$-action with proper moment map, and
 \item \label{item:nondeg} every tall singular point is non-degenerate and no such point has a hyperbolic block and  connected $T$-stabilizer.   
   \end{enumerate}
Unfortunately,  these criteria are fairly restrictive;  see Theorem~\ref{thm:too_bad}.
 Our main goal in this paper is to find a larger class of integrable systems that has connected fibers by weakening the non-degeneracy assumption above. To achieve this, we introduce ``ephemeral" degenerate singular points (see Definition~\ref{eph} below), examples of which have appeared in the literature in the context of $p \! : \! -q$ resonances (see \cite{SD} and references therein), and special Lagrangian fibrations (see \cite[Chapter III.3.A]{HL} and \cite{J,CB}).

To state our results, we need some definitions from integrable systems, following \cite[Section 1]{ST1}.

\begin{Definition}\label{int_sys}
  An {\bf integrable system} is a quadruple $\left(M,\omega, V,
    f\right)$, where $\sm$ is a connected
  $2n$-dimensional symplectic manifold, $V$ is an $n$-dimensional real
  vector space, and $f \colon M \to V^*$ is a smooth map  
   such that:
  \begin{enumerate}
  \item  the Poisson bracket $\big\{ \langle f, v_1 \rangle, \langle f, v_2 \rangle \big\} = 0$ for all $v_1, v_2 \in V$, and
  \item the function $f$ is regular on a dense subset of $M$.
  \end{enumerate}
  Here $\langle \cdot, \cdot \rangle \colon V^* \times V \to \R$ is the natural pairing.
\end{Definition} 
Throughout this paper, we assume that every integrable system is {\bf complete}, i.e., it is a Hamiltonian $V$-space. 

Fix an integrable system $\left(M,\omega, V,
    f\right)$. A point $p \in M$ is {\bf
  regular} if $D_pf$ has full rank and {\bf singular} otherwise. Since the $V$-orbit $\rorbit := V \cdot p$ through any point $p \in M$ is isotropic,  the vector space $\left(T_p
  \rorbit\right)^{\omega}/T_p\rorbit$ inherits a natural symplectic structure. 
Since the stabilizer $W \subseteq V$ of $p$ acts by linear symplectomorphisms on $T_p M$ and fixes $T_p \rorbit,$
it acts on $(T_p\rorbit)^\omega/T_p \rorbit$ by linear symplectomorphisms; we call this
action the {\bf linearized action} of $W$ at $p$.

\begin{Definition} \label{defn:non-degenerate}
Let $p$ be a point in an integrable system $\left(M,\omega,V,f\right)$.
Let $W_0 \subseteq V$ be the Lie algebra of the stabilizer $W$ of $p$.  The point $p$ is {\bf non-degenerate}\footnote{This is consistent with the standard definition (see \cite[Remark 1.3]{ST1}).} if there exist identifications $ W_0 \simeq \R^{\dim W}$ and $(T_p\rorbit)^\omega/T_p \rorbit \simeq \C^{\dim W}$ such that the homogeneous moment map for the linearized action of $W$ at $p$ is a product of the following ``blocks":
\begin{itemize}
    \item ({\bf elliptic}) $z \in \C \mapsto \frac{1}{2}|z|^2$;
    \item ({\bf hyperbolic}) $z \in \C \mapsto \Im(z^2)$;
    \item ({\bf focus-focus}) $(z_1,z_2) \in \C^2 \mapsto \left(\frac{1}{2}(|z_1|^2 - |z_2|^2),\Im(z_1z_2)\right)$.
\end{itemize}
Here $\Im(w)$ denotes the imaginary part of $w \in \C$ and $\C^{\dim W}$ has the standard symplectic structure.  Moreover, $p$ has {\bf purely elliptic type} if $p$ is non-degenerate and all the blocks are elliptic.
\end{Definition}

In particular, by definition regular points are non-degenerate and have purely elliptic type.

Next we recall some fundamental facts about complexity one $T$-spaces. Fix a  {\bf complexity one $\boldsymbol{T}$-space}, that is,
a triple $(M,\omega,\Phi)$, where $(M,\omega)$ is a connected $2n$-dimensional
symplectic manifold
and $\Phi \colon M \to \ft^*$ is a moment map for 
an effective Hamiltonian $T$-action.

 Since the $T$-orbit $\mathfrak O := T \cdot p$ through any point $p \in M$
is isotropic, the vector space 
$(T_p \torbit)^{\omega}/T_p\torbit$ inherits a natural symplectic structure.
Moreover,  the natural action
of the stabilizer $H \subseteq T$ of $p$ on $T_p M$ induces a symplectic representation of $H$ on
$(T_p \torbit)^{\omega}/T_p\torbit$, called the {\bf symplectic slice representation} at $p$.
It is isomorphic to the representation of $H$ on $\C^{h+1}$
associated to some injective homomorphism  $\rho \colon H \hookrightarrow (S^1)^{h+1}$, where $h = \dim H$.
Let $\fh$ be the Lie algebra of $H$ and let $\eta_i \in \fh^*$ be the differential of the $i$'th component of $\rho$.
The  (not necessarily distinct) vectors $\eta_0,\dots,\eta_h$  are
called the {\bf isotropy weights} at $p$; the multiset of isotropy weights at $p$ is
uniquely determined.


The  {\bf reduced space} at $\beta \in \ft^*$ is the quotient $\Phi^{-1}(\beta)/T$; we set $\MmodT := \Phi^{-1}(0)/T$. 
A point $p \in M$ is {\bf short} if $[p]$ is an isolated point in  $\Phi^{-1}(\Phi(p))/T$ and is {\bf tall} otherwise.

As we mentioned in the beginning of this introduction, in \cite{ST1} we proved that every $2n$-dimensional integrable system that satisfies both condition \ref{item:extend} and \ref{item:nondeg} above has connected fibers.
Unfortunately, as the following result shows, these conditions are fairly restrictive.

\begin{Theorem}\label{thm:too_bad}
Let $(M,\omega, \Phi)$ be a complexity one $T$-space with a proper moment map. Assume that some fiber of $\Phi$ contains (at least) three distinct orbits $\mathfrak O_1, \mathfrak O_2, \mathfrak O_3$ such that, for all $p \in \mathfrak O_i$:

\begin{enumerate}
\item \label{item:0.8a} If one of the isotropy weights at $p$ is zero, then  the  stabilizer of $p$ has more than two components; and
\item \label{item:0.8b} If there is a pair $\eta, \eta'$ of isotropy weights  at $p$ such that $\eta + \eta' = 0$, then
 the  stabilizer of $p$ is not connected.
\end{enumerate}
 Then there is no integrable system of the form $\left(M, \omega, \ft \times\R, f:= (\Phi, g) \right)$ such that every tall singular point is non-degenerate and each fiber of $f$ is connected.
\end{Theorem}

\begin{Remark}\label{rmk:thm-too-bad-interior}
    If the fiber of $\Phi$  in Theorem \ref{thm:too_bad} lies over the interior of $\Phi(M)$, then the local model for $\mathfrak O_i$, as defined below, has a surjective moment map for all $i$ (see \cite[Theorem 6.5]{Sja}). Hence, conditions  \eqref{item:0.8a} and \eqref{item:0.8b} above can be replaced with the following (see Remark \ref{rmk:interior}):
    \begin{enumerate}[label=(\arabic*)$'$]
\item  The stabilizer of  $p$ is neither trivial nor isomorphic to $\Z_2$; and
\item  If the stabilizer of  $p$  is isomorphic to $S^1$, then the  set of  isotropy weights is not $\{-1,+1\}$.
\end{enumerate}
\end{Remark}

 Many complexity one $T$-actions have a fiber containing three distinct orbits that satisfy \eqref{item:0.8a} and \eqref{item:0.8b} above. Theorem \ref{thm:too_bad} immediately implies that they cannot be extended to integrable systems satisfying the assumptions (A) and (B) above -- if they did, then each fiber would be connected (cf. \cite{HSS,HSSS}).

To overcome this limitation, we introduce some degenerate singular points to allow more integrable systems that extend complexity one $T$-spaces and have connected fibers. From now on we fix an inner product on the Lie algebra $\ft$. Consider a point $p$ in a complexity one $T$-space with stabilizer $H$ and symplectic slice representation $\rho \colon H \hookrightarrow (S^1)^{h+1}$. Set $Y := T \times_H \fh^{\circ} \times \C^{h+1}$, where $\fh^{\circ}$ is the annihilator of $\fh$ in $\ft^*$ and $H$ acts on $T$ by multiplication, on $\fh^{\circ}$ trivially, and on $\C^{h+1}$ by $h_0 \cdot z = \rho(h_0^{-1})z$.
There exists a  canonical symplectic form $\omega_Y$ on $Y$ so that the $T$-action $s \cdot [t,\alpha,z] = [st,\alpha,z]$
has homogeneous moment
map $\Phi_Y \colon Y \to \ft^*$ given by 
$$\textstyle \Phi_Y([t,\alpha,z]) = \alpha + \Phi_H(z),$$
where $\Phi_H(z) = \frac{1}{2} \sum_i  \eta_i |z_i|^2$ is the homogeneous moment
map for the $H$-action on $\C^{h+1}$ (see, e.g., \cite[Remark 2.2]{ST1}). Here, we use the inner product to identify $\ft^*$ with $\fh^{\circ} \oplus \fh^*$. By the Marle-Guillemin-Sternberg local normal form theorem  \cite{GS,M}, there is
a $T$-equivariant symplectomorphism from a $T$-invariant neighborhood of
$p$
to an open subset of $Y$ that takes $p$ to $[1,0,0]$.  We call $(Y,\omega_Y,\Phi_Y)$ the {\bf local model} for $p$,  and say that $Y$ is {\bf tall} if $[1,0,0]$ is tall.

Let $Y$ be a local model\footnote{Here, and throughout the paper, all the local models that we consider are complexity one.} with moment map $\Phi_Y$, and consider the reduced space $$\YmodT:=\Phi^{-1}_Y(0)/T = (T \times_H \{0\} \times \Phi^{-1}_H(0))/T \simeq \Phi^{-1}_H(0)/H =: \C^{h+1} /\! / \! H.$$  
Each $T$-invariant smooth function $g \colon Y \to \R$ induces a function $\overline{g} \colon \YmodT \to \R$.
Moreover, if $R \colon \C^{h+1} \hookrightarrow T \times_H \fh^\circ \times \C^{h+1} $ denotes the inclusion $z \mapsto [1,0,z]$, then $R^*g \colon \C^{h+1} \to \R$ is an $H$-invariant function. Given $\ell \in \N:= \{n \in \Z \mid n \geq 0\}$, we regard the degree $\ell$ Taylor polynomial of $R^*g$ at $0$
as an $H$-invariant function $T_{0}^\ell (R^*g) \colon \C^{h+1} \to \R$.
By a slight abuse of terminology, we define the degree $\ell$ {\bf reduced Taylor polynomial} of $g$ at $p =[1,0,0] \in Y$ to be the function $\overline{T^\ell_p g}  \colon \YmodT \to \R$ defined by 
$$ \overline{T^\ell_p g}([t,0,z]):= T_{0}^\ell (R^* g)(z) \quad \text{for all } [t,0,z] \in \YmodT$$
for all $\ell \in \N$, and by $0$ for all $\ell \in \Z \smallsetminus \N$.

If  the local model $Y$ is tall, then there
 is a unique $\xi = (\xi_0,\dots,\xi_h)
\in \Z_{\geq 0}^{h+1}$ so that the map $P \colon Y \to \C$ given by $[t,\alpha,z] \mapsto
\prod_{j=0}^h z_j^{\xi_j}$ is well-defined and induces a homeomorphism between the reduced space
$\YmodT$ and $\C$. We call $P$ the {\bf defining monomial} of $Y$ and  the sum $N := \sum_{j=0}^h \xi_j$ the {\bf degree} of $P$. Note that $N=1$ if and only if all the nearby orbits in $\Phi_Y^{-1}(0)$ have the same stabilizer; otherwise, 
$[1,0,0]$ has a strictly larger stabilizer. (See Section \ref{section:complexity_one} for more details.) 

\begin{Remark}\label{rmk:large-N}
   Let $(M,\omega,\Phi)$ be a complexity one $T$-space. A tall point $p \in M$ satisfies conditions \eqref{item:0.8a} and \eqref{item:0.8b} of Theorem~\ref{thm:too_bad} exactly if its local model has a defining monomial of degree $N > 2$;  see Lemma \ref{pureveryexceptional}.
\end{Remark}

\begin{Example}\label{ex1}
Let $\Z_N \varsubsetneq S^1$ act on $\C$ by $\lambda \cdot z = \lambda z$. The associated local model  $Y =  S^1 \times_{\Z_N} \R^* \times \C$ is tall.
The defining monomial $P\colon \C \to \C$ is given by $P(z) = z^N$ and has degree $N$.
\end{Example}

\begin{Example} \label{ex2}
Let $S^1$ act effectively on $ Y =  \C^2$ by  $\lambda \cdot (z_1,z_2) = (\lambda^p z_1, \lambda^{q} z_2)$, where $p, q \in \Z$, with moment map $\Phi (x,y)= \frac{p}{2} |z_1|^2 + \frac{q}{2} |z_2|^2$. Assume that $Y$ is a tall local model, which occurs exactly if $pq \leq 0$.
%
Then the defining monomial $P \colon \C^2 \to \C$ is given by $P(x,y) = z_1^{|q|} z_2^{|p|}$ and has degree $N := |p| + |q|$.
\end{Example}

\begin{Remark}
    \label{rmk:proto-ephemeral}
    Let $Y$ be given as in Example~\ref{ex1} or \ref{ex2}, and let  $g \colon Y \to \R$ be the imaginary part of the defining monomial
$P$.
It is straightforward to check that if  $N=2$,
 then   $(Y,\omega_Y, \ft \times \R, (\Phi_Y,g))$ is an integrable
system that satisfies assumption (B) from the first paragraph of this introduction. In particular, the tall singular point $p:= [1,0,0] \in Y$ is non-degenerate and has one hyperbolic block  and disconnected $T$-stabilizer in Example~\ref{ex1}, and one focus-focus block in  Example~\ref{ex2}.
Unfortunately, if $N > 2$ these examples violate assumption (B) because $p$ is a tall degenerate singular point; cf. Remark \ref{rmk:large-N}. 
\end{Remark}

 We would like to also allow integrable systems with singular points modeled on the examples with $N > 2$ constructed in Remark \ref{rmk:proto-ephemeral}. 
For example,  the singular points based on Example~\ref{ex2} are known in the integrable systems literature as  $p \! : \! -q$ resonances (see \cite{E,NSZ,SD}, Example \ref{ex:modelephemeral}, and Remark \ref{rmk:examples-eph-lit}). With this motivation, we introduce the following class of singular points.


\mute{Alternate wording for Remark \ref{rmk:proto-ephemeral}:
It is straightforward to check that if $N = 2$ then   $(Y,\omega_Y, \ft \times \R, (\Phi_Y,g))$ is an integrable
system that satisfies assumption (B) of \cite[Theorem 1.6]{ST1}. (For example, if $N = 2$ then $p:= [1,0,0] \in Y$ has a focus-focus block for $Y$ as in  Example~\ref{ex1}, and a hyperbolic block but disconnected $T$-stabilizer for Example~\ref{ex2}.)   We would like to also include the  $N > 2$ cases;  unfortunately, these violate  assumption (B)  because $p$ is a degenerate singular point. Nevertheless, all these examples have certain properties in common: }

\begin{Definition}\label{eph}
Let $(M,\omega,\Phi)$ be a complexity one $T$-space and let $g \colon M \to \R$ be a $T$-invariant  smooth function.  Fix $p \in M$.  By subtracting a constant, we may assume that $g(p) = 0$. Then $p$ is an {\bf ephemeral (singular) point} of $\left(M, \omega, \ft \times \R, f:= (\Phi, g) \right)$ if it is tall, the  defining monomial of its local model $Y$  has degree $N > 1$, and there exists a $T$-equivariant symplectomorphism from a $T$-invariant neighborhood of $p$ to an open subset of $Y$ taking $p$ to $[1,0,0]$ such
that, under this identification,
  \begin{enumerate}[label=(\alph*),ref=(\alph*)]
  \item \label{item:a} the reduced Taylor polynomial $\overline{T^{N-1}_p g} \colon \YmodT \to \R$ vanishes identically, and
\item \label{item:b} the zero set of the reduced Taylor polynomial $\overline{T^N_p g} \colon \YmodT \to \R$
is homeomorphic to $\R$.
  \end{enumerate}
\end{Definition}

As we show in Section \ref{section:jets}, properties \ref{item:a} and \ref{item:b} in Definition \ref{eph} hold for {\em one} such identification with a local model if and only if they hold for {\em every} such identification.

\begin{Example}\label{ex:modelephemeral}
Let $(Y,\omega_Y, \Phi_Y)$ be a tall local model. Let $g \colon Y \to \R$ be the imaginary part of the defining monomial $P \colon Y \to \C$. 
If the degree $N$ of $P$ is greater than $1$ then $[1,0,0]$ is an ephemeral singular point of  $(Y, \omega_Y, \ft \times \R, f:= (\Phi_Y,g)). $  To see this, first note that  $T_p^{N-1} (R^*g) = 0$, and so the reduced Taylor polynomial $\overline{T_p^{N-1} g}$ is identically $0$.
Moreover, since $P$ induces a homeomorphism from $\YmodT$ to $\C$,
the zero set of $\overline{T_p^N g} = \overline{g}$ is homeomorphic to $\R$.  Indeed,  $\YmodT$ is not naturally a smooth manifold, but if we use $\overline{P}$ to identify $\YmodT$ with $\C$ then  $\overline{g}( x + i y) =   y$, and so $\overline{g}$ is a smooth function with no critical points.
\end{Example}

%

\begin{Remark}
Let $\left(M,\omega, \ft \times \R, f:=(\Phi,g)\right)$ be an integrable system such that $(M,\omega,\Phi)$ is a complexity one $T$-space. Every ephemeral point is singular by Lemma~\ref{greg2}. Additionally, by Lemmas ~\ref{short_elliptic} and \ref{lemma:eph-ell} a non-degenerate singular point $p$ is ephemeral exactly if it either has a focus-focus block, or has a hyperbolic block and disconnected $T$-stabilizer.
\end{Remark}

\begin{Remark}\label{rmk:examples-eph-lit}
    Examples of ephemeral singular points have appeared in various works in the literature.  As we mentioned, the fixed points in integrable systems with two degrees of freedom known as $p \! : \! -q$ resonances are ephemeral singular points; see \cite{E,NSZ,SD}. As these authors show, in four-dimensions a phenomenon known as  ``fractional monodromy" appears near $p \! : \! -q$ resonances;  in higher dimensions,  it is natural to ask when an analogous phenomenon occurs near ephemeral points. Additionally, the construction of singular special Lagrangian fibrations in the study of the SYZ conjecture involves certain singular points known as the Harvey-Lawson singularities. (See \cite[III.3.A]{HL} for the definition of these singular points, and \cite{CB,J} for their role in the above problem.) In dimension 6 and higher, these singular points are degenerate; it is straightforward to check that they are ephemeral.
\end{Remark}

We can now state our main result. 

\begin{Theorem}\label{thm:main}
Let $\left(M,\omega, \ft \times \R, f:=(\Phi,g)\right)$ be an integrable system such that $(M,\omega,\Phi)$ is a complexity one $T$-space with a proper moment map. Assume that each tall
singular point in $\left(M,\omega, \ft \times \R, f\right)$ is either  non-degenerate or ephemeral.  The following are equivalent:
\begin{enumerate}
\item \label{item:0.13a} No tall non-degenerate singular point has a hyperbolic block and connected $T$-stabilizer.
\item \label{item:0.13b} The fiber $f^{-1}(\beta,c)$ is connected and the reduced space\footnote{By \cite{Li}, if $M$ is compact then we can replace the phrase ``the reduced space $\Phi^{-1}(\beta)/T$" in \eqref{item:0.13b} by ``$M$" and eliminate ``for each $(\beta, c) \in \ft^* \times \R$."} $\Phi^{-1}(\beta)/T$ is  simply connected
for each $(\beta, c) \in \ft^* \times \R$.
\end{enumerate}
\end{Theorem}

We can also prove that generically, under the above assumptions, the integrable system has connected fibers if and only if no  tall non-degenerate singular point has a hyperbolic block and connected $T$-stabilizer (see Proposition~\ref{prop:converse}; cf. \cite[Theorem 1.11]{ST1}). Finally, we construct a family of examples that shows that our main theorem meaningfully extends previous results (see Proposition \ref{prop:example}).

 We believe that, as the results of this paper suggest,
 ephemeral degenerate singular point are tractable, but at the same time including them greatly increases the set of examples  of integrable systems  we can consider.

\mute{02/13/24: Morse theory does not give an obstruction to the existence of an integrable system with connected fibers and only non-degenerate and ephemeral points. We believe that we can construct examples (starting from dimension four?). 
}

As in \cite{ST1}, the fundamental idea behind the proof of
Theorem \ref{thm:main}, as well as of Theorem \ref{thm:too_bad},  is
that each reduced space $\Phi^{-1}(\beta)/T$ containing more than one point 
can be given the structure of a smooth closed, connected oriented surface so that $\overline{g} \colon \Phi^{-1}(\beta)/T \to \R$ is a Morse function (see Proposition \ref{prop:Morse}); moreover,
\begin{enumerate}[label=(\roman*)]
\item \label{item:hyp} $T$-orbits of  non-degenerate singular points with a hyperbolic block and connected $T$-stabilizer correspond to critical points of index $1$; 
\item \label{item:pep} $T$-orbits of singular points  that have  purely elliptic type  and are ``critical points of $g$ modulo $\Phi$"\footnote{
In particular, all singular points  over the interior of $\Phi(M)$ are critical point of $g$ modulo $\Phi$; see Definition~\ref{critmodphi} and \cite[Remark 3.5]{ST1}.} correspond to critical points of index $0$ or $2$; and
\item \label{item:regular} the remaining $T$-orbits, including those of ephemeral singular points, correspond to regular points. (This justifies the name ``ephemeral".)
\end{enumerate}

Additionally, as we show in the proof of Theorem \ref{thm:too_bad}, all tall points that satisfy conditions \eqref{item:0.8a} and \eqref{item:0.8b} of Theorem \ref{thm:too_bad} satisfy condition \ref{item:pep} above. Once these facts are established, the proofs reduce to standard results in Morse theory. 

\subsection*{Structure of the paper} In Section \ref{section:complexity_one} we recall some facts from \cite{KT1, ST1}, including  the  key result, Proposition~\ref{connected}, which allows us to construct Morse functions on the reduced spaces under certain conditions. In Sections \ref{section:non-deg} and \ref{sec:ephemeral_points}, respectively, we prove that non-degenerate and ephemeral points satisfy the technical conditions of this key result, as well as establish properties \ref{item:hyp}, \ref{item:pep} and \ref{item:regular} above. In Section \ref{sec:proofs} we prove our main results. We continue our study of ephemeral points in Section \ref{section:jets}, where characterize them {\em intrinsically}. Finally, in Section \ref{sec:examples} we construct a family of examples of integrable systems that satisfy the hypotheses of Theorem \ref{thm:main} and have degenerate ephemeral points.

\subsection*{Acknowledgments} D. Sepe was partially supported by FAPERJ grant JCNE E-26/202.913/2019 and by a CAPES/Alexander von Humboldt Fellowship for Experienced Researchers 88881.512955/2020-01. This study was financed in part by the Coordenação de Aperfeiçoamento de Pessoal de Nível Superior -- Brazil (CAPES) -- Finance code 001.  S. Tolman was partial supported by NSF  DMS Award 2204359 and Simons Foundation Collaboration Grant 637995.

\section{Morse theory on reduced spaces}\label{section:complexity_one}

In this section we first recall some facts about tall (complexity one) local models due to Karshon and Tolman in \cite{KT1}. 
We then review the method developed in \cite{ST1},  which allows us to construct Morse functions on the reduced spaces  under certain conditions (see Proposition \ref{connected}). This eventually allows us to use Morse theory on surfaces to prove our main results.

The first result from \cite{KT1} we recall is \cite[Lemma 2.3]{ST1}. 

\mute{This is used in Definition \ref{defn:defining_poly}, in Lemma~\ref{veryexceptional}, Lemma~\ref{lemma:vanishing-poly}, in the fluff of Section \ref{sec:examples} and in Proposition \ref{prop:example}.}
\begin{Lemma}\label{lemma:defn_poly}
 Let $Y = T \times_H \fh^{\circ} \times \C^{h+1}$ be a local model associated to an injective homomorphism $\rho \colon H \hookrightarrow (S^1)^{h+1}$. Then there exists $\xi = (\xi_0,\ldots, \xi_h) \in \Z^{h+1}$ such that the homomorphism $(S^1)^{h+1} \to S^1$
given by $\lambda \mapsto
\prod\limits_{j=0}^h \lambda_j^{\xi_j}$ induces a short exact sequence
\begin{equation}
  \label{eq:3}
  1 \to H \stackrel{\rho}{\to} (S^1)^{h+1} \to S^1 \to 1;
\end{equation}
moreover, the vector $\xi$ is unique up to sign. Finally,  $\xi_i\xi_j \geq 0$ for all $i,j$ if and only if the local model $Y$ is tall.
\end{Lemma}

\mute{We don't explicitly refer the definition anywhere but clearly we need it.}
Lemma \ref{lemma:defn_poly} allows us to introduce the ``defining monomial" of a tall local model, following \cite[Definition 5.12]{KT1} and \cite[Definition 2.6]{ST1}.

\begin{Definition}\label{defn:defining_poly}
 Fix a tall local model $Y$. By Lemma~\ref{lemma:defn_poly}, there exists a unique $\xi = (\xi_0,\ldots, \xi_h) \in \Z_{\geq 0} ^{h+1}$ inducing the short exact sequence \eqref{eq:3}. The map $P \colon \C^{h+1} \to \C$ given by $P(z) = \prod_{j=0}^hz_j^{\xi_j}$ is called the {\bf defining monomial} of $Y$. By a slight abuse of notation, we  use the same name for the map $P\colon Y \to \C$ given by $[t,\alpha,z]\mapsto
  \prod_{j=0}^hz_j^{\xi_j}$.
\end{Definition}

 The defining monomial of a tall local model $Y$ allows to describe the symplectic quotient $\YmodT = \Phi^{-1}_Y(0)/T$. Following \cite{KT1}, given a complexity one $T$-space $(M,\omega,\Phi)$, we endow the reduced space $\MmodT$  with the subquotient topology and define the  sheaf of {\bf smooth functions} on $\MmodT$ as follows:  Given an open subset $U \subseteq \MmodT$, let $C^\infty(U)$ be the set of functions $\overline{g} \colon U \to \R$ whose pullback to $\Phi^{-1}(0)$ is the restriction of a
smooth $T$-invariant function $g \colon M \to \R$.  Moreover, a {\bf diffeomorphism} between topological spaces endowed with sheaves of smooth functions is a {\bf homeomorphism} that induces an isomorphism between the sheaves of smooth functions. The next result from \cite{KT1} we recall is a simplification of \cite[Lemma 2.10]{ST1}.

\mute{This is used in Proposition \ref{ephemeral} and Proposition \ref{prop:example}.}
\begin{Lemma}\label{trivial}
  Let $Y = T \times_H \fh^\circ \times \C^{h+1}$ be a tall local model
  with moment map $\Phi_Y \colon Y \to \ft^*$. The defining monomial $P \colon Y \to \C$ induces a homeomorphism $ \overline{P} \colon \YmodT \to \C$ that restricts to a diffeomorphism on the
  complement of $\left\{[1,0,0]\right\}$.
\end{Lemma}

To state our main result, Proposition~\ref{connected}, we switch our attention to integrable
systems extending complexity one $T$-spaces, and recall a few notions and results from \cite[Section 3]{ST1}. Let $(M,\omega, \ft \times \R, f:=(\Phi,g))$ be an integrable
    system such that $\left(M,\omega,\Phi\right)$ is a complexity
  one $T$-space. As in the Introduction, given $p \in M$, the function $g$ induces a function
  $\overline{g} \colon \Phi^{-1}(\Phi(p))/T \to \R.$ We start by recalling \cite[Definition 3.1]{ST1}.

\mute{We don't explicitly  refer to this definition anywhere but clearly we need it.}
\begin{Definition}\label{critmodphi} 
  A point $p \in M$ is a {\bf critical point of $\boldsymbol{g}$ modulo $\boldsymbol{\Phi}$} if
  $D_p g$ vanishes on $\ker D_p\Phi$; otherwise, $p$ is a {\bf regular point of $\boldsymbol{g}$ modulo $\boldsymbol{\Phi}$.}
\end{Definition}

 Near each regular point of $g$ modulo $\Phi$ we can endow the reduced space with a structure of smooth manifold so that $\overline{g}$ is smooth with no critical points.

\mute{The following result is used in Remark 1.9, Corollary \ref{cor:tricho}, the proof of Theorem \ref{thm:too_bad}, and Proposition \ref{prop:example}, as well as in the intro. We only need the "moreover", but we decided to keep the whole statement.}
\begin{Lemma}[\cite{ST1}, Lemma 3.4]\label{greg2}
Let $(M,\omega, \ft \times \R, f:=(\Phi,g))$ be an integrable
    system such that $\left(M,\omega,\Phi\right)$ is a complexity
  one $T$-space. Let $p \in M$ be a regular point of $g$ modulo $\Phi$. Then there exists an open neighborhood $U_p$ of $[p]$ in the reduced space
$\Phi^{-1}(\Phi(p))/T$ and a map $\Psi_p \colon U_p \to \R^2$ satisfying:
\begin{enumerate}[label=(\alph*),ref=(\alph*)]
    \item \label{item:-4} $\Psi_p \colon U_p \to \Psi(U_p)$ is a diffeomorphism onto an open set; and
    \item\label{item:-2} $\overline{g} \circ \Psi_p^{-1} \colon \Psi_p(U_p) \to \R$ is a smooth function with no critical points.
\end{enumerate}
Moreover, $p$ is tall, has purely elliptic type, and  the defining monomial of the local model for $p$ has degree $1$. \end{Lemma}

 By Lemma \ref{greg2}, in order to construct Morse functions on the reduced spaces, we need to impose some conditions near the critical points of $g$ modulo $\Phi$. More precisely, the following key result holds.

\mute{The following result is used in Proposition \ref{prop:Morse} and referred to in the fluff.}
\begin{Proposition}[\cite{ST1}, Proposition 3.6] \label{connected}
Let $(M,\omega, \ft \times \R, f:=(\Phi,g))$ be an integrable system such that $\left(M,\omega,\Phi\right)$ is a complexity
  one $T$-space.  Fix $\beta \in \ft^*$.
Assume that for each critical point $p \in \Phi^{-1}(\beta)$ of $g$ modulo $\Phi$  
 there exists an open neighborhood $U_p$ of $[p]$ in  the reduced space $\Phi^{-1}(\beta)/T$  
and a map $\Psi_p \colon U_p \to \R^2$ satisfying:
\begin{enumerate}
\item $\Psi_p \colon U_p \to \Psi_p(U_p)$ is a homeomorphism onto an open set;
\item  $\Psi_p$ restricts to a diffeomorphism on $U_p \smallsetminus \{[p]\}$; and
\item $ \overline{g}\circ \Psi_p^{-1} \colon \Psi_p(U_p) \to \R$ is a Morse function.
\end{enumerate}
Then  $\Phi^{-1}(\beta)/T$ can be given the structure of a smooth oriented surface such that $\overline g \colon \Phi^{-1}(\beta)/T \to \R$ is a Morse function.
Moreover, $[p]$ is a critical point of $\overline{g}$ of index $\mu$ exactly if both $p$ is a critical point of $g$ modulo $\Phi$ and $\Psi_p([p])$ is a critical point of $\overline g \circ \Psi_p^{-1}$ of index $\mu$.
\end{Proposition}

\section{Non-degenerate singular points}\label{section:non-deg}

In this section, we study  non-degenerate singular points in integrable systems that extend complexity one $T$-actions. Our main goal 
is to recall results from \cite{ST1} that show how non-degenerate singular
points fit within the framework of Morse theory on reduced spaces
of Proposition \ref{connected} (see Propositions \ref{gcrit} and \ref{nondegenerate}).
Additionally, we determine which non-degenerate singular points are ephemeral in the sense of Definition \ref{eph} (see Corollary \ref{cor:tricho}).

The first claim that we recall 
 characterizes short non-degenerate singular points.

\mute{The following is used in the proof of Theorem~\ref{thm:main}, 
Lemma~\ref{veryexceptional}, and Corollary~\ref{cor:tricho}.}

\begin{Lemma}[\cite{ST1}, Lemma 6.2]\label{short_elliptic}
Let $(M,\omega, \ft \times \R, f:=(\Phi,g))$ be an integrable system such that $\left(M,\omega,\Phi\right)$ is a complexity one $T$-space. If $p \in M$ is short and non-degenerate, then $p$ has purely elliptic type.
\end{Lemma}

The next two statements show that tall non-degenerate singular points  satisfy the hypotheses of Proposition \ref{connected}. The first deals with points that have purely elliptic type; the second deals with all other cases.

\mute{ The following is used in Proposition \ref{prop:Morse}.}
\begin{Proposition} [\cite{ST1}, Proposition 4.4] \label{gcrit}
Let $(M,\omega,\ft \times \R, f := (\Phi,g))$
be an integrable system
so that  $(M,\omega,\Phi)$ is a complexity one $T$-space.
Let $p \in \Phi\inv(0) \cap g^{-1}(0)$ be a tall 
critical point
of $g$ modulo $\Phi$ that has purely elliptic type. After possibly replacing $g$ by $-g$, there exist an open neighborhood $U$ of $[p]$ in $\MmodT$
and a map $\Psi \colon U \to \R^2$ taking $[p]$ to $(0,0)$ satisfying:
\begin{enumerate}
\item $\Psi \colon U \to \Psi(U)$ is a homeomorphism onto an open set;
\item $\Psi$ restricts to a diffeomorphism on $U \smallsetminus \{[p]\}$; and
\item $(\overline{g} \circ \Psi^{-1})(x,y)= x^2 + y^2$ for all $(x,y) \in \Psi(U)$.
\end{enumerate}
\end{Proposition}

\mute{ The following is used in Proposition \ref{prop:Morse}.}

\begin{Proposition}[\cite{ST1}, Proposition 6.3]\label{nondegenerate}
Let $\left(M,\omega, \ft \times \R, f:=\left(\Phi,g\right)\right)$ be an integrable  system such that $\left(M,\omega,\Phi\right)$ is a complexity one $T$-space. Assume that  $p \in \Phi^{-1}(0) \cap g^{-1}(0)$ is non-degenerate and tall; let $N$ be the degree of  the defining monomial of the local model for $p$.
Then $p$  has at most one non-elliptic block; furthermore,
\begin{enumerate}
\item If $p$ has a hyperbolic block and connected $T$-stabilizer, then $N = 1$. Moreover, there exist an open neighborhood $U$ of $[p] \in \MmodT$ and a map $\Psi \colon U \to \R^2$ taking $[p]$ to $(0,0)$ satisfying:
\begin{enumerate}[label=(\alph*),ref=(\alph*)]
    \item \label{item:-500} $\Psi \colon U \to \Psi(U)$ is a diffeomorphism onto an open set; and
    \item \label{item:-600} $(\overline{g} \circ \Psi^{-1})(x,y) = x^2 - y^2$ for all $(x,y) \in \Psi(U)$.
   \end{enumerate}
    \item If  $p$ either has a focus-focus block, or  has a hyperbolic block and disconnected $T$-stabilizer, then $N = 2$. Moreover,  there exist an open neighborhood $U$ of $[p] \in \MmodT$ 
  and a map $\Psi \colon U \to \R^2$ taking $[p]$ to $(0,0)$ satisfying:
  \begin{enumerate}
  \item \label{item:800} $\Psi \colon U \to \Psi(U)$ is a homeomorphism onto an open set;
  \item \label{item:1200} $\Psi$ restricts to a diffeomorphism on $U \smallsetminus \{[p]\}$; and
      \item \label{item:1300} $(\overline{g} \circ \Psi^{-1})(x,y)= y$ for all $(x,y) \in \Psi(U)$.
  \end{enumerate}
  \end{enumerate}
\end{Proposition}

 In this paper we do not directly quote the second sentence of part (2) of Proposition~\ref{nondegenerate}. Instead, we show below that in these cases the points are, in fact, ephemeral (see Corollary \ref{cor:tricho}); in the next section we show that that sentence holds for all ephemeral points (see Proposition \ref{ephemeral}).  The final claim that we recall allows us to determine which non-degenerate singular points are ephemeral.

\mute{10/23/24: The following is  used in Corollary \ref{cor:tricho}.} 

\begin{Lemma}[\cite{ST1}, Lemma 6.7]\label{lemma:eph-ell}
    Let $(Y,\omega_Y,\ft \times \R, f:=(\Phi_Y,g))$ be an integrable system, where  $Y = T \times_H \fh^{\circ} \times \C^{h+1}$ is a tall local model.
    Assume that $p = [1,0,0]$ is a non-degenerate critical point of $g$ modulo $\Phi$ with $g(p) = 0$. If $R \colon \C^{h+1} \hookrightarrow Y$ is the inclusion $z \mapsto [1,0,z]$, then
\begin{itemize}
    \item the map $\YmodT \to \R$ taking $[t,0,z] \in \YmodT$ to  $T^1_0 (R^*g)(z)$ is identically zero, and
    \item the zero set of the map $\YmodT \to \R$ taking $[t,0,z] \in \YmodT$ to  $T^2_0 (R^*g)(z)$ is
    \begin{itemize}
        \item 
     a single point if $p$ has purely elliptic type,
     \item  homeomorphic to $\{(x,y) \in \R^2 \mid xy=0\}$ if $p$ has a hyperbolic block and connected $T$-stabilizer,
    and  \item homeomorphic to $\R$ otherwise.
    \end{itemize}
\end{itemize}
\end{Lemma}

 We conclude this section with the following trichotomy.

\mute{The following is used in Proposition \ref{prop:Morse}.}
\begin{Corollary}\label{cor:tricho}
Let $\left(M,\omega,\ft \times \R, f=\left(\Phi,g\right)\right)$ be an integrable system such that $\left(M,\omega,\Phi\right)$ is
a complexity  one $T$-space. Let $p \in M$ be a non-degenerate singular point.
Then exactly one of the following holds: $p$ has purely elliptic type,  $p$ has a hyperbolic block and connected $T$-stabilizer, or $p$ is ephemeral.
\end{Corollary}

\begin{proof} By Lemmas \ref{greg2} and \ref{short_elliptic}, it suffices to consider the case that $p$ is a tall critical point of $g$ modulo $\Phi$.  
So by the Marle-Guillemin-Sternberg local normal form, we may assume that $M$ is a tall local model $Y = T \times_H \fh^\circ \times \C^{h+1}$, that $\Phi = \Phi_Y$, and that $p =
[1,0,0]$. If $p$ does not have elliptic type and $p$ does not have a hyperbolic block and connected $T$-stabilizer, then the degree $N$ of the defining monomial of $Y$ is $2$ by Proposition \ref{nondegenerate}. Hence, the result follows from Lemma \ref{lemma:eph-ell}.
\end{proof}
\section{Ephemeral points}\label{sec:ephemeral_points}

\mute{In this section we refer to/use: Definition~\ref{eph}, Remark \ref{rmk:proto-ephemeral}, Lemma \ref{trivial}, Proposition \ref{connected}.}
In this section, we study ephemeral points in
an integrable system $(M,\omega, \ft \times \R, f:=(\Phi,g))$ that extends a complexity one $T$-action.
We prove that, although every  ephemeral point $p \in M$ is singular, we can put a smooth structure
on the reduced space near $[p]$ so that the induced function $\overline{g}$ is smooth and has no critical points; see Proposition~\ref{ephemeral}.
Hence, the singularity is transitory; this explains the terminology ``ephemeral" in Definition~\ref{eph}. In particular, every ephemeral point satisfies the hypotheses of Proposition \ref{connected}.

Before proving this claim, we need a few definitions.
Let $Y = T \times_H \fh^\circ \times \C^{h+1}$ be a local model.
By a slight abuse of notation, we call a $T$-invariant function $g \colon Y \to \R$ a {\bf $\boldsymbol{T}$-invariant (homogeneous) polynomial} if its restriction to $\fh^\circ \times \C^{h+1}$ is a (homogeneous) polynomial
 (see \cite[Definition 2.13]{ST1}).
More generally, recall that if $g \colon Y \to \R$ is any smooth $T$-invariant function,  the degree $\ell$ reduced Taylor polynomial at $p=[1,0,0] \in Y$ is the function $\overline{T^{\ell}_p g} \colon \YmodT \to \R$ defined by 
$$ \overline{T^{\ell}_p g}(\llbracket t,0,z \rrbracket) = T^{\ell}_0(R^*g)(z) \quad \text{for all } \llbracket t,0,z \rrbracket \in \YmodT.$$
Here,  $R \colon \C^{h+1} \hookrightarrow Y$ denotes the inclusion $z \mapsto [1,0,z]$.

\begin{Lemma}\label{lemma:invariant2}
Let $Y = T \times_H  \fh^{\circ} \times \C^{h+1}$ be a tall 
local model with moment map $\Phi_Y \colon Y \to \ft^*$ and defining monomial $P \colon \C^{h+1} \to \R$ of degree $N$.
Given a $T$-invariant smooth function $g \colon Y \to \R$,
 there exists a smooth function $h \colon \R^3 \to \R$ such that,
for all $\llbracket t,0,z \rrbracket \in \YmodT$,
$$
        \overline{g}( \llbracket t,0,z \rrbracket) = h(\Re P(z),\Im P(z),|z|^2).
$$
For all $\ell
\in \Z_{\geq 0}$, the degree $\ell$ reduced Taylor polynomial of $g$ is
  $$\overline{T_p^\ell g} (\llbracket t,0,z \rrbracket) =  
\hspace{-.15in}
\mathlarger{\mathlarger{\sum}}_{\mathsmaller{Ni + Nj + 2k \leq \ell}} \hspace{-.1in}
\frac{1}{i! j! k!} 
 \frac{\del^{i+j+k}h}{\del x^i \del y^j \del \rho^k}(0) \,
  (\Re P(z))^i (\Im P(z))^j |z|^{2k}
$$
for all  $ \llbracket t,0,z \rrbracket  \in \YmodT$.  Finally, if $\overline{T^{\ell - 1}_p g} \colon \YmodT \to \R$ vanishes identically, we may  further assume
that  $\frac{\del^{i+j+k}h}{\del x^i \del y^j \del \rho^k}(0) = 0$ for all $i, j, k \in \Z_{\geq 0}$ so that $Ni + Nj + 2k < \ell.$
\end{Lemma}

\begin{proof}
Let $\Pi \colon Y \to \R^3 \oplus \ft^*$ denote the map 
$$\Pi([t,\alpha,z]) = (\Re P(z), \Im P(z), |z|^2,
\Phi_Y([t,\alpha,z])). $$
\noindent By \cite[Lemma 2.14]{ST1}, the algebra of $T$-invariant polynomials on $Y$
is generated by the real and imaginary parts of the defining monomial $P \colon Y \to \C$, the map $[t,\alpha,z] \mapsto |z|^2$, and the components of  $\Phi_Y$. Hence, since $g$ is $T$-invariant, by Schwarz's
theorem applied to the $H$-action on $\fh^\circ \times \C^{h+1}$  there exists a smooth function $ \breve{h} \colon \R^3 \oplus \ft^* \to
\R$ such that $g = \Pi^* \breve{h}$; see \cite{Sch}.
Let $h \colon \R^3 \to \R$ be the restriction of $\breve{h}$ to  $\R^3 \oplus \{0\}$, that is,
let $h(x,y,\rho) := \breve{h}(x,y,\rho,0)$. Then $ \overline{g}( \llbracket t,0, z \rrbracket ]) 
= h(\Re P(z),\Im P(z),|z|^2)$ for all $ \llbracket t,0,z \rrbracket \in \YmodT,$ as required.

We introduce a grading on polynomials on $\R^3 \oplus \ft^* \simeq \R^3 \oplus \fh^\circ \oplus \fh^*$, so that  the first two coordinates of $\R^3$ have degree $N$, the third coordinate has degree $2$,  the components of $\fh^\circ$ have degree $1$, and the components of $\fh^*$ have degree $2$. Since each component of $\Pi$ is a
$T$-invariant polynomial on $Y$ of  the corresponding degree, for any integer $m$  the pullback under $\Pi$ of the Taylor polynomial of $\breve{h}$ of graded degree $m$ at $0$ is the $T$-invariant polynomial of degree $m$ on $Y$ whose
restriction to $\fh^\circ \times \C^{h+1}$ is the Taylor polynomial  $T^m_0(R^*g)$. 
Moreover, let
\begin{equation*}
  \sum_{Ni + Nj + 2k \leq  m}
        \frac{1}{i!j!k!}\frac{\del^{i+j+k}h}{\del x^i \del y^j \del \rho^k}(0) x^i y^j \rho^k 
\end{equation*}
\noindent
be the Taylor polynomial of  $h$ of graded degree $m$  at $0 \in \R^3$. Note that it is
the restriction to $\R^3 \simeq \R^3 \oplus \{0\}$
of the Taylor polynomial of  $\breve{h}$ of graded degree $m$ at $0 \in \R^3 \oplus \ft^*$.
Therefore, for all $ \llbracket t,0,z \rrbracket \in \YmodT$,
\begin{equation*}
       \overline{T_p^{m} g }( \llbracket t,0,z \rrbracket ) 
    = \sum_{Ni + Nj + 2k \leq m} 
\frac{1}{i!j!k!}\frac{\del^{i+j+k}h}{\del x^i \del y^j \del \rho^k}(0)\,
(\Re P(z))^i (\Im P(z))^j |z|^{2k}.
\end{equation*}
 Finally, if $\overline{T^{\ell-1}_p g}$ vanishes on $\YmodT$, then  for
all $ \llbracket t,0,z \rrbracket \in \YmodT$,
\begin{equation*}
  \sum_{Ni + Nj + 2k < \ell} \frac{1}{i!j!k!}\frac{\del^{i+j+k}h}{\del x^i \del y^j \del \rho^k}(0)\,
  (\Re P(z))^i (\Im P(z))^j |z|^{2k} = 0. 
\end{equation*}
\noindent
Hence, if we replace $h$ by  
$$h(x,y,\rho) - \sum_{Ni + Nj + 2k < \ell}\frac{1}{i!j!k!}\frac{\del^{i+j+k}h}{\del x^i \del y^j \del \rho^k}(0)\, x^i y^j \rho^k,$$
\noindent 
then this function (which we also denote by $h$) satisfies
$\frac{\del^{i+j+k}h}{\del x^i \del y^j \del \rho^k}(0) = 0$ for all $i, j, k \in \Z_{\geq 0}$ so that $Ni + Nj + 2k < \ell$, and still satisfies the remaining properties.
\end{proof}

We can now prove the main result of this section, which generalizes Example \ref{ex:modelephemeral}.

\begin{Proposition}\label{ephemeral}
Let $(M,\omega, \ft \times \R, f:=(\Phi,g))$ be an integrable system such that $(M,\omega,\Phi)$ is a complexity one $T$-space.  Let $p \in \Phi^{-1}(0) \cap g^{-1}(0)$ be an ephemeral point.
Then there exist an open neighborhood $U$ of $[p] \in \MmodT$ 
  and a map $\Psi \colon U \to \R^2$ taking $[p]$ to $(0,0)$ satisfying:
  \begin{enumerate}
  \item \label{item:8} $\Psi \colon U \to \Psi(U)$ is a homeomorphism onto an open set;
  \item \label{item:12} $\Psi$ restricts to a diffeomorphism on $U \smallsetminus \{[p]\}$; and
      \item \label{item:13} $(\overline{g} \circ \Psi^{-1})(x,y)= y$ for all $(x,y) \in \Psi(U)$.
  \end{enumerate}
\end{Proposition}

\begin{proof}
By the Marle-Guillemin-Sternberg local normal form theorem, we may assume that $M$ is a tall local model $Y = T \times_H \fh^\circ \times \C^{h+1}$, that $\Phi = \Phi_Y$, and that $p =
[1,0,0]$. Let $N$ be the degree of the associated defining monomial $P \colon \C^{h+1} \to \C$. Since $p$ is ephemeral, we may further assume that the reduced Taylor polynomial $\overline{T^{N-1}_p g} \colon \YmodT \to \R$ vanishes identically and that the zero set of $\overline{T^N_p g} \colon \YmodT \to \R$
is homeomorphic to $\R$. 

Set $m := \lceil N/2 \rceil$. By Lemma \ref{lemma:invariant2}, there exists a smooth function $h \colon \R^3 \to \R$ 
such that:
\begin{enumerate}[label=(\roman*),ref=(\roman*)]
\item \label{item:taylor}For all $ \llbracket t,0,z \rrbracket \in \YmodT$,
\begin{equation*}
\begin{split}
 \overline{g}(\llbracket t,0,z \rrbracket) &= h(\Re P(z),\Im P(z),|z|^2), \quad \text{and} \\
\overline{T^N_pg} (\llbracket t,0,z \rrbracket)   & =
\begin{cases} 
\frac{\del h}{\del x}(0)\Re P(z) + \frac{\del h}{\del y}(0) \Im P(z) 
& \mbox{if $N$ is odd,} \\
\frac{\del h}{\del x}(0)\Re P(z) + \frac{\del h}{\del y}(0) \Im P(z) 
+ \frac{1}{m!}\frac{\del^m h}{\del
  \rho^m}(0) |z|^{2m} & \mbox{if $N$ is even.}
\end{cases}
\end{split}
\end{equation*}
\item For each integer $0 \leq k < m$, $\frac{\del^k h}{\del \rho^k}(0) = 0$.
\end{enumerate}

The group $(S^1)^{h+1}$ acts on $Y$ by $T$-equivariant symplectomorphisms that preserve $\Phi_Y$, and hence acts on $\YmodT$. If we replace $\overline{g}$ by its pullback under such an automorphism $\lambda \in (S^1)^{h+1}$, we also replace $h$ by its pullback under an automorphism of $\R^3$ that rotates the first two coordinates by $P(\lambda) \in S^1$. Hence, since  $P((S^1)^{h+1}) = S^1$, we may further
assume that $\frac{\del h}{\del y}(0) \geq 0$ and $\frac{\del h}{\del x}(0)= 0$. 

Define the map $\Psi \colon \YmodT \to \R^2$ by 
$$\Psi(\llbracket t,0,z \rrbracket) = (\Re P(z), \overline{g}(\llbracket t,0,z \rrbracket)).$$
By definition,  $\Psi$  satisfies condition \eqref{item:13} in the statement
for any $U \subseteq \YmodT.$
By \cite[Lemma 2.12]{ST1} there exists $C > 0$ so that 
\begin{equation}\label{Ceq}
|z|^2 = C|P(z)|^\frac{2}{N} \text{ for all } \llbracket t,0,z \rrbracket \in \YmodT.
\end{equation}
Define a continuous function $\check{h} \colon \R^2 \to \R$  by
\begin{equation}
  \label{eq:11}
  \check{h}(x,y) := h(x,y,C(x^2+y^2)^\frac{1}{N}).
  \end{equation}
By item \ref{item:taylor} above and \eqref{Ceq},
\begin{equation}\label{eq:psi} \Psi(\llbracket t,0,z \rrbracket) = (\Re P(z), H(\Re P(z),\Im P(z)))\text{ for all } \llbracket t,0,z \rrbracket \in \YmodT.
\end{equation} 
Our aim is to show that there exists a convex open neighborhood $V$ of $(0,0) \in \R^2$
such that $\frac{\partial \check h}{\partial y} > 0$ on $V \smallsetminus \{(0,0)\}$.
Suppose that this claim holds.
Since $V$ is convex,
the intersection $V \cap (\{x_0\} \times \R)$ is connected 
for all $x_0 \in \R$.
Hence,  since $ \check{h} (0,0) = 0$,
a careful application of the Mean Value Theorem shows that  $\check{h} (x_0,y) < \check{h} (x_0,y')$ for
all  $(x_0,y)$ and $(x_0,y')$  in $V$ with $y < y'$.
Thus, the map $(x,y) \mapsto (x,\check{h} (x,y))$ is a continuous injection from $V$ to $\R^2$.  Therefore, by Invariance of Domain, the image of $V$ is open in $\R^2$ and the above map is a homeomorphism onto its image.  Since the map $(x,y) \mapsto (x,\check{h} (x,y))$ is a submersion on $V \smallsetminus \{(0,0)\}$, it is a diffeomorphism
from $V \smallsetminus \{(0,0)\}$ to its image.
Moreover, by  Lemma \ref{trivial}, the map $\overline{P} \colon \YmodT \to \C$ is a homeomorphism that restricts to a diffeomorphism on the complement of $[p]$. So by \eqref{eq:psi},
there exists an open neighborhood $U$ of $[p]$ in $\YmodT$ such that  $\Psi \colon U \to \Psi(U)$
satisfies conditions \eqref{item:8} and \eqref{item:12} in the statement.  Hence, it remains to show that the map $(x,y) \mapsto (x,\check{h}  (x,y))$ has the desired properties.

Since $\frac{\del h}{\del x}(0)= 0$, item \ref{item:taylor} above and \eqref{Ceq} imply that, under the identification given by $\overline{P} \colon \YmodT \to \C$,
\begin{equation*}
  \overline{T^N_pg}(x + i y) =
  \begin{cases}
    \frac{\del h}{\del y}(0) y & \mbox{if $N$ is odd}, \\
    \frac{\del h}{\del y}(0) y +  \frac{C^m}{m!} \frac{\del^m h}{\del
  \rho^m}(0) (x^2 + y^2)^\frac{1}{2} & \mbox{if $N$
      is even},
\end{cases}
\end{equation*}
for all $x+ i y \in \C$.
Since the zero set of 
$\overline{T^N_pg}$ is homeomorphic to $\R$ and since $\frac{\del h}{\del y}(0) \geq 0$, we obtain that $\frac{\del h}{\del y}(0) > 0$ if $N$ is odd; by using polar coordinates to calculate the zero set, we also obtain that
$\frac{\del h}{\del y}(0) >  \frac{C^m }{m!} \big| \frac{\del^m h}{\del \rho^m}(0) \big|$ if $N = 2m$.

Since $\frac{ \del^k h}{\del \rho^k}(0) = 0$ for all $0 \leq k < m$,
by applying Lemma \ref{trick} below $m$ times
and by combining terms as necessary, 
there exist smooth functions $h_x,h_y,h_\rho$ on $\R^3$  such that
\begin{equation}
\label{eq:10}
h= xh_x + yh_y + \rho^m h_{\rho}.
\end{equation}

By combining \eqref{eq:11} with \eqref{eq:10}, on $\R^2 \smallsetminus \{(0,0)\}$ we have that
\begin{equation*}
    \begin{split}
        \frac{\del \check{h} }{\del y} =
h_y &+ \frac{2 m C^m y}{N (x^2+y^2)^\frac{N-m}{N}}
h_\rho + x \frac{\del h_x}{\del y} + 
y \frac{\del h_y}{\del y} +C^m(x^2+y^2)^\frac{m}{N} \frac{\del h_\rho}{\del y} \\
&+ \frac{2 Cy }{N  (x^2+y^2)^{\frac{N-1}{N} }}\left( x \frac{\del h_x}{\del \rho} 
+ y \frac{\del h_x}{\del \rho} + C^m(x^2+y^2)^\frac{m}{N}  \frac{\del h_\rho}{\del \rho}\right),
    \end{split}
\end{equation*}
\noindent
where $\frac{\del \check{h} }{\del y}$ is evaluated at $(x,y) \in \R^2 \smallsetminus \{(0,0)\}$, and $h_x, h_y,
h_{\rho}$ and their partial derivatives are evaluated at $(x,y, C
(x^2+y^2)^\frac{1}{N})$. If $N$ is odd, then $2m = N + 1$, and so each term except the first approaches $0$ as $(x,y)$ approaches $(0,0)$. Therefore,  $\frac{\del \check{h} }{\del y}$ approaches $h_y(0) = \frac{ \partial h}{\partial y}(0)$, which is positive.
If $N$ is even, then $2m = N$, and so the first term approaches $h_y(0)$, the second term is equal to $C^m y(x^2 + y^2)^{-\frac{1}{2}}  h_\rho$, and each of  the remaining terms approaches $0$. 
Moreover, $h_y(0) = \frac{\partial h}{\partial y}(0)$ is greater than $|C^m h_\rho(0)| =  \frac{C^m}{m!} |\frac{\partial^m h}{\partial \rho^m}(0)|$.  Hence, in either case there exists a convex neighborhood $V$
of $(0,0)$ so that
$\frac{\del \check{h} }{\del y}$
is positive on $V \smallsetminus \{(0,0)\}$, as required.
\end{proof}

Above, we used the following result,
proved in \cite[Lemma 2.1]{Mil}.

\begin{Lemma}\label{trick}
  Let $V \subseteq \R^n$ be a convex neighborhood of $0$ and let $h \colon V
  \to \R$ be a smooth function. If $h(0) = 0$, then there exist smooth
  functions $h_1,\ldots, h_n \colon V \to \R$ such that, for all
  $(x_1,\ldots,x_n) \in V$, 
  $$ h(x_1,\ldots,x_n) = \sum\limits_{i=1}^n x_i h_i(x_1,\ldots,x_n). $$
\end{Lemma}

\section{Proofs of the main results}\label{sec:proofs}

\mute{In this section we use: \ref{connected}, \ref{gcrit}, \ref{nondegenerate}, \ref{ephemeral}, \ref{cor:tricho}, \ref{thm:connected}, \ref{veryexceptional}, \ref{greg2}.}
 In this section, we prove our main results, Theorems \ref{thm:too_bad} and \ref{thm:main};  we also state and prove a related result, Proposition \ref{prop:converse}. Consider an integrable system that extends a complexity one $T$-space. Assume that each tall singular point is either non-degenerate or ephemeral.
The fundamental idea of this section is that the results of Sections~\ref{section:complexity_one}, \ref{section:non-deg}, and \ref{sec:ephemeral_points} combined show  
that each component of the reduced space containing more than one point 
can be given the structure of a smooth oriented surface so that the induced function
$\overline{g}$ is a Morse function.
More precisely, we have the following generalization of \cite[Proposition 7.1]{ST1} that includes ephemeral points.

\mute{Proposition \ref{prop:Morse} is used in the proofs of Theorems \ref{thm:too_bad} and \ref{thm:main}.}
\begin{Proposition}\label{prop:Morse}
Let $(M, \omega, \ft \times \R, f = (\Phi,g))$ be an integrable system such that $(M,\omega,\Phi)$ is a complexity one $T$-space.  Assume that each tall 
singular point in $(M,\omega,\ft \times \R, f)$ is either non-degenerate or ephemeral. Fix $\beta \in \ft^*$.
Each component of the reduced space $\Phi^{-1}(\beta)/T$  containing more than one point
  can be given the structure of a smooth  oriented surface such that $\overline{g} \colon \Phi^{-1}(\beta)/T \to \R$ is a Morse function. Moreover, given $p \in \Phi^{-1}(\beta)$:
\begin{enumerate}[label=(\roman*)]
\item $[p]$ is a critical point of $\overline{g}$ of index $1$  exactly if $p$ is a  non-degenerate singular point with a hyperbolic block and
connected $T$-stabilizer,
\item $[p]$ is a critical point of $\overline{g}$ of index $0$ or $2$ exactly if $p$ is a critical point of $g$ modulo $\Phi$ with purely elliptic type, and
\item $[p]$ is a regular point of $\overline{g}$ exactly if $p$ is ephemeral or is a regular point of $g$ modulo $\Phi$.
\end{enumerate}
\end{Proposition}

\begin{proof}

 We may assume that $\beta = 0$. For simplicity, assume that $\MmodT$ is connected and contains more than one point.    Then, by definition, every point in $\Phi^{-1}(0)$ is tall.  If $p \in \Phi^{-1}(0)$ is a critical point of $g$ modulo $\Phi$,
then $p$ is also a singular point of $\left(M,\omega, \ft \times \R, f\right)$. So, by assumption, $p$ is either non-degenerate or ephemeral. Hence, by Corollary \ref{cor:tricho}, $p$ either has purely elliptic type, is non-degenerate with a hyperbolic block and connected $T$-stabilizer, or is ephemeral. Therefore, by Proposition \ref{connected}, the claim follows  from Propositions  \ref{gcrit}, \ref{nondegenerate} and  \ref{ephemeral}.
\end{proof}

We also  need the following well-known result, which is proved in \cite{At,GS2,Sja, LMTW}.

\mute{Theorem~\ref{thm:connected} is used in Theorems \ref{thm:too_bad} and \ref{thm:main}.}
\begin{Theorem}\label{thm:connected}
Let $T$ act on a connected symplectic manifold  $(M,\omega)$ with  moment map $\Phi \colon M \to \ft^*$.
If $\Phi$ is proper, then the fibers of $\Phi$ are connected and $\Phi$ is open as a map to $\Phi(M)$.
\end{Theorem}

 Finally, we need the following result, which characterizes those tall local models with a defining monomial of degree greater than two. 

\begin{Lemma} \label{pureveryexceptional}
Let $Y = T \times_H \fh^\circ \times \C^{h+1}$ be a tall local model,  where $H$ acts on $\C^{h+1}$ with weights $\eta_0,\dots,\eta_h$. Then the degree $N$ of the defining monomial of $Y$ is greater than two exactly if:
\begin{enumerate}[label=(\arabic*),ref=(\arabic*),leftmargin=*]
    \item \label{item:one-weight-zero} If  $\eta_i = 0$ for some $i$, then $H$ has more than two  components; and 
    \item \label{item:sum-two-weights} If  $\eta_i + \eta_j = 0$ for some $i, j$, then $H$ is not connected.
\end{enumerate}
\end{Lemma}

\begin{proof} 
Let  $P(z) = \prod_{i=0}^h z_i^{\xi_i}$ be the defining monomial of $Y$.
 By Lemma~\ref{lemma:defn_poly}, the image of $H$ in $(S^1)^{h+1}$ is the kernel of the homomorphism $ (S^1)^{h+1} \to S^1$ given by $\lambda \mapsto \prod_{i} \lambda_i^{\xi_i}$; in particular,  $\sum_{i} \xi_i \eta_i  = 0$ is the unique linear relation among the $\eta_i$. 
The weight $\eta_i = 0$ for some $i$ exactly if  $P(z) = z_i^{\xi_i}$; in this case,
$P$ has degree $\xi_i$ and $H$ has $\xi_i$ components. 
Similarly, the sum $\eta_i + \eta_j = 0$ for some $i, j$ exactly if  
$P(z) = z_i^{\xi_i} z_j^{\xi_j}$ and $\xi_i = \xi_j$; in this case, $P$ has degree $N = 2 \xi_i$ and $H$ has $\xi_i$ components.
In any other case, $N = \sum_i \xi_i > 2$.
The claim follows immediately.
\end{proof}

\mute{Remark \ref{rmk:interior} is used in Remark \ref{rmk:thm-too-bad-interior}.}
\begin{Remark}\label{rmk:interior} If $Y$ is a tall local model with a surjective moment map, then  by \cite[Lemma 5.2]{KT1}  we can
choose the $\xi$ in Lemma~\ref{lemma:defn_poly} so that $\xi_i > 0 $ for all $i$.
Therefore, if the degree $N:=\sum_i \xi_i$ of the defining monomial is $1$ then $H$ is trivial, while if $N = 2$ then either $H \simeq \Z_2$, or $H \simeq S^1$ and the set of
isotropy weights are $\{-1,+1\}$.
\end{Remark}

We can now prove our main results, Theorems \ref{thm:too_bad} and \ref{thm:main}.

\begin{proof}[Proof of Theorem \ref{thm:too_bad}]
Let $(M, \omega, \ft \times\R, f:= (\Phi, g) )$ be an integrable system such that $(M,\omega, \Phi)$ is a complexity one $T$-space with a proper moment map. We may assume that the fiber $\Phi^{-1}(0)$ contains three distinct $T$-orbits $\mathfrak{O}_1,\mathfrak{O}_2, \mathfrak{O}_3$ that satisfy properties \eqref{item:0.8a} and \eqref{item:0.8b} in the statement of the theorem. Furthermore, assume that  every tall singular point in $\left(M, \omega, \ft \times\R, f\right)$ is non-degenerate.  Since $\Phi$ is proper, the reduced space  $\MmodT$ is compact and connected by Theorem~\ref{thm:connected}; this implies that every point in $\mathfrak O_i$ is tall. So by  the Marle-Guillemin-Sternberg local normal form theorem and Lemma~\ref{pureveryexceptional} every point in each
of the orbits $\mathfrak{O}_i$ has a tall local model whose defining monomial has degree $N > 2$. Therefore, each such point is a critical point of $g$ modulo $\Phi$ by Lemma \ref{greg2}, and  has purely elliptic type by Proposition \ref{nondegenerate}.
Hence, by Proposition \ref{prop:Morse} (or by \cite[Proposition 7.1]{ST1}), $\MmodT$ can be given the structure of a smooth closed, connected, oriented surface so that $\overline{g} \colon \MmodT \to \R$ is a Morse function; moreover, each $\mathfrak{O}_i$ corresponds to a critical point of $\overline{g}$ of index $0$ or $2$. By, for example, \cite[Lemma 7.2]{ST1}, every Morse function on a  closed, connected oriented manifold with connected fibers has a unique critical point of index or coindex 0. Hence, not all fibers of $\overline{g}$ are connected, and so not all fibers of $f$ are connected.
\end{proof}

\begin{proof}[Proof of Theorem \ref{thm:main}]
Let $(M,\omega, \ft \times \R, f:=(\Phi,g))$ be an integrable system such that $(M,\omega,\Phi)$ is a complexity one $T$-space with a proper moment map. Assume that each tall 
singular point in $\left(M,\omega, \ft \times \R, f\right)$ is either  non-degenerate or ephemeral. Fix $\beta \in \ft^*$; we may assume that $\beta = 0$.
Since $\Phi$ is proper, the reduced space  $\MmodT$ is compact and connected by Theorem~\ref{thm:connected}. 

If the reduced space $\MmodT$ is a single point, then it contains no tall orbits; moreover, it is simply connected and the fiber $f^{-1}(0,c)$ is connected for all $c \in \R$.

Otherwise, every point in $\MmodT$ is tall, and by Proposition~\ref{prop:Morse}, $\MmodT$ can be given the structure of  a smooth closed,  connected oriented surface  so that 
$\overline{g} \colon \MmodT \to \R$
is a Morse function; moreover, the point $[p] \in \MmodT$ associated to $p \in \Phi^{-1}(0)$ is a critical point of $\overline{g}$ of index $1$  exactly if $p$ is a non-degenerate singular point with a hyperbolic block and connected $T$-stabilizer.
Since $T$ is connected, the fiber $f^{-1}(0,c)$ is connected if and only if $\overline{g}^{-1}(c)$ is connected. 

Hence, the result follows from \cite[Lemma 7.3]{ST1}, which states that a Morse function on a closed, connected 2-dimensional manifold has no critical point of index one exactly if the manifold is simply connected and each fiber is connected.
\end{proof}

\begin{Remark}\label{rmk:more_general}
More generally, Theorems \ref{thm:too_bad} and \ref{thm:main} both hold if $\Phi$ is proper as a map to an open convex subset of $\ft^*$, because the conclusions of Theorem~\ref{thm:connected} still hold in this case by \cite[Theorem 4.3]{LMTW}. 
\end{Remark}

Finally, we prove that, generically, the integrable systems we consider have connected fibers exactly if they have no tall non-degenerate singular points with a hyperbolic block and connected $T$-stabilizer.
Note that, by \cite[Lemma 5.7 and Corollary 9.7]{KT1}, all the reduced spaces of a complexity one $T$-space with a proper moment map have the same genus.
Hence, Theorem~\ref{thm:main} and Proposition~\ref{prop:converse} together show that \cite[Theorem 1.11]{ST1} can be extended to allow tall ephemeral points, instead of only allowing tall non-degenerate  singular points.

\begin{Proposition}\label{prop:converse}
    Let $(M,\omega, \ft \times \R, f:=(\Phi,g))$ be an integrable system such that $(M,\omega,\Phi)$ is a complexity one $T$-space with a proper moment map. Assume that each tall 
singular point in $\left(M,\omega, \ft \times \R, f\right)$ is either non-degenerate or ephemeral. If each fiber of $f$ contains at most one $T$-orbit of  tall non-degenerate singular points with a hyperbolic block, then $f^{-1}(\beta,c)$ is connected for all $(\beta,c) \in \ft^* \times \R$ if and only if $M$ has no  tall non-degenerate singular points with a hyperbolic block and connected $T$-stabilizer.
\end{Proposition}

\begin{proof}
If $M$ has no tall non-degenerate singular points with a hyperbolic block and connected $T$-stabilizer, then by Theorem \ref{thm:main} the fiber $f^{-1}(\beta,c)$ is connected for each $(\beta,c) \in \ft^* \times \R$. Hence, we may assume that each fiber of $f$ is connected.

Fix $\beta \in \ft^*$  such that $\Phi^{-1}(\beta)/T$ contains more than one point; we may assume that $\beta = 0$. Since $\Phi$ is proper, the reduced space  $\MmodT$ is compact and connected by Theorem~\ref{thm:connected}.  By  Proposition~\ref{prop:Morse}, $\MmodT$ can be given
the structure of a smooth closed, connected oriented surface so that $\overline{g} \colon \MmodT \to \R$ is a Morse function; moreover, the point $[p] \in \MmodT$ associated to $p \in \Phi^{-1}(0)$ is a critical point of $\overline{g}$ of index $1$  exactly if $p$ is a non-degenerate singular point with a hyperbolic block and connected $T$-stabilizer.  
 Since $T$ and the fiber $f^{-1}(0,c)$ are connected,  the fiber $\overline{g}^{-1}(c)$ is connected for each $c \in \R$. Hence the result follows from \cite[Lemma 7.4]{ST1}, which states that a Morse function on a closed, connected oriented surface with connected fibers has the property that each fiber contains an even number of critical points of index 1.
\end{proof}

\section{An intrinsic characterization of ephemeral points}\label{section:jets}

\mute{03/05/24: Warning: In general, there may be polynomials vanishing on $V_{\Phi}$
  that do not lie in $I_{\Phi}$. Consider, for instance, the case of
  $M = \C^2$ and $\Phi = |z_1|^2 + |z_2|^2$. Then $V_{\Phi} = \{0\}$
  but there are polynomials that vanish on $V_{\Phi}$ that do not lie
  in $I_{\Phi} = \langle |z_1|^2 + |z_2|^2 \rangle$. However, this does not happen in our case. At one point we considered mentioning something about this issue/the real nullstellensatz.
  }

\mute{In this section we use: Definition~\ref{eph}.}

The definition of an ephemeral point in the Introduction depends on choosing a symplectic identification of a $T$-invariant neighborhood of the point with an open subset of the associated local model.
In this section, we give an {\em intrinsic} characterization of an ephemeral point, i.e., one that does not depend on any such choice; in particular, we show that if  a point is ephemeral  then conditions \ref{item:a} and \ref{item:b} of Definition~\ref{eph} hold for {\em every} such identification  (see Proposition \ref{cor:criterion} below). Finally, we provide tools to check whether a point is ephemeral using any coordinate chart in Lemma \ref{lemma:calculate}.

We start by introducing the intrinsic notions that we need. Fix $\ell \in \N := \{n \in \Z \mid n \geq 0\}$. Suppose that $g \in C^{\infty}(\R^n)$ {\bf vanishes below order $\boldsymbol \ell$ at $\boldsymbol 0$}, i.e., all partial derivatives of $g$ of order strictly less than $\ell$ vanish at $0$.
Then  $T^\ell_0 g$, the  degree $\ell$ Taylor polynomial  of $g$  at $0$, is a homogeneous polynomial of degree $\ell$ on $T_0 \R^n$. Moreover, if $\psi\colon \R^m \to \R^n$ is  smooth and $\psi(0) = 0$, then by the chain rule
\begin{enumerate}
    \item \label{item:natural} $\psi^*g \in C^{\infty}(\R^m)$ vanishes below order $\ell$ at $0$, and
    \item \label{item:natural_poly} $T^{\ell}_0 (\psi^*g) = (D_0\psi)^* (T^{\ell}_0 g)$.
\end{enumerate}

Next fix a point $p$ in a manifold $M$. Given a coordinate chart $\varphi$ on $M$ centered at $p$, we say that $g$ {\bf vanishes below order $\boldsymbol{\ell}$ at $\boldsymbol p$} if $g \circ \varphi^{-1}$ vanishes below order $\ell$ at $0$. By property \eqref{item:natural} above, this does not depend on the choice of  $\varphi$. Let $Q^{\ell}_p \subseteq C^\infty(M)$ denote  the ideal of functions that vanish below order $\ell$ at $p$, and let $S^{\ell}(T_p M)$ be the vector space of homogeneous polynomials of degree $\ell$ on $T_p M$. Given $g \in Q^{\ell}_p$, the {\bf jet of $\boldsymbol{g}$ of order $\boldsymbol{\ell}$  at $\boldsymbol{p}$} is
\begin{equation}
\label{eqn:defn_jet}
  j^{\ell}_p g:= (D_p \varphi)^*(T^{\ell}_0(g \circ \varphi^{-1})) \in S^\ell(T_pM),  
\end{equation}
where $\varphi$ is a coordinate chart centered at $p$.
By property \eqref{item:natural_poly} above, this polynomial does not depend on the choice of $\varphi$. Moreover, by properties \eqref{item:natural} and \eqref{item:natural_poly}, these notions are natural; given a smooth map $\psi \colon N \to M$ 
with $\psi(q) = p$: 
\begin{equation} \label{eq:natural}
\text{If }g \in Q^{\ell}_p  \text{ then }\psi^*g \in Q^{\ell}_q \text{ and }
 j^{\ell}_q(\psi^*g) = (D_q\psi)^*(j^\ell_pg).
 \end{equation}

\begin{Remark} 
The map  $j^{\ell}_p\colon Q^{\ell}_p \to S^\ell(T_pM)$ is a surjective homomorphism of vector spaces with kernel 
$Q^{\ell +1}_p$.
Under the resulting identification of $Q^\ell_p/Q^{\ell+1}_p$ with $S^\ell(T_pM)$, the homomorphism $j_p^\ell$ is the restriction of the standard jet map $C^\infty(M) \to C^\infty(M)/Q^{\ell+1}_p$ to the subspace $Q^{\ell}_p \subseteq C^\infty(M)$. This motivates our use of the term ``jet''.
\end{Remark}

Let $\mathcal{I}$ be a subspace of $C^{\infty}(M)$. We say that $g \in C^{\infty}(M)$ {\bf vanishes below order $\boldsymbol{\ell}$ at $\boldsymbol{p}$ modulo} $\bm{ \mathcal{I}}$ if $g \in \mathcal{I} + Q_p^{\ell}$, i.e., $g - h \in Q^{\ell}_p$ for some $h \in \mathcal{I}$. Moreover, the {\bf homogenization of $\bm{\mathcal I}$ at $\boldsymbol{p}$} is a graded subspace of the graded polynomial algebra $S(T_pM) := \oplus_{\ell \in \N} S^\ell(T_pM)$; it is the image of the map $\bigoplus_{\ell \in \N} \mathcal I \cap Q^{\ell}_p  \to S(T_p M) $  that sends $g \in \mathcal I \cap Q^{\ell}_p$ to $j^{\ell}_p g$. The {\bf zero locus of the  homogenization of $\bm{\mathcal{I}}$ at $\boldsymbol{p}$} is the set of points $V_{\mathcal{I},p}$ in $T_p M$ where every polynomial in  the homogenization of $\mathcal{I}$ at $p$ vanishes. By \eqref{eq:natural},
given a diffeomorphism $\varphi \colon N \to M$  with $\varphi(q) = p$:
\begin{equation} \label{eq:natural2}
D_q \varphi (V_{\varphi^*\mathcal I,q}) = V_{\mathcal I,p}.
\end{equation}

\begin{Example}
Let $\mathcal{I}$ be the ideal in $C^{\infty}(M)$ generated by a function $f \in Q_p^\ell \smallsetminus Q_p^{\ell + 1}$. It is straightforward to check that the homogenization of $\mathcal{I}$ at $p$ is the ideal in $S(T_p M)$ generated by $j_p^\ell f$; its zero locus $V_{\mathcal{I},p} \subseteq T_pM$ is the zero set of $j_p^\ell f$.
\end{Example}

Now assume that a torus $T$ acts on $M$.  Then the stabilizer $H$ of $p$ acts linearly on the smooth slice  $W := T_p M/T_p \mathfrak O$,
where $\mathfrak{O}$ is the $T$-orbit through $p$.
By the smooth slice theorem for proper actions (see, e.g., \cite[Theorem 2.4.1]{DK}), we may locally identify $M$  with $Y = T \times_H W$ and  $p$ with $[1,0]$.
Define  $r \colon W \to Y$ by $r(v)= [1,w]$, $\chi \colon \fh^\perp \times W \to Y$  by $\chi(\eta,w)= [\exp(\eta), w]$, and $\pi \colon \fh^\perp \times W \to W$ by $\pi(\eta,w) = w$. Since $D_{(0,0)} \chi$ is
invertible, there is a coordinate chart $\varphi$  centered at $p$  with inverse $\chi$.
Using $\varphi$ to identify $T_p Y$ with $\fh^\perp \times W$ gives $T_p \mathfrak O =\fh^\perp \times \{0\}$ and $D_0 r(w) = (0,w)$ for all $w \in W \simeq T_0 W$.

If $g \in C^{\infty}(Y)$ is $T$-invariant, then $\chi^* g$ is constant along the fibers of $\pi$. Thus $\chi^* g = \pi^* (r^* g)$, and
so  $T^{\ell}_{(0,0)} (\chi^* g) = \pi^* (T^\ell_0 (r^* g))$ for all $\ell \in \N$.
 Hence, 
\begin{equation}
\label{eq:vanish-r}
    g \in Q^{\ell}_p \text{ exactly if } r^*g \in Q^{\ell}_0. 
\end{equation} 
Moreover, in this case
$j_p^\ell g = \pi^* (j^\ell_0 (r^* g))$, and so   $j_p^\ell g$ is invariant  under addition
by elements of $T_p \mathfrak O$.
Additionally, by \eqref{eq:natural} the jet $j^{\ell}_p g$ is invariant under the slice representation of $H$.

If $\mathcal{I}$ is a subspace of the algebra of $T$-invariant  functions in $C^{\infty}(Y)$, then
its pullback $r^* \mathcal I$ is a subspace of the algebra of $H$-invariant functions in $C^{\infty}(V)$.
By the preceding paragraph, $V_{\mathcal{I},p}$ is closed under addition by elements in $T_p \mathfrak O$. Indeed, $V_{\mathcal I,p} = \fh^\perp \times V_{r^* \mathcal I,0}$ and so 
\begin{equation} 
\label{eq:isomorphism}
V_{\mathcal I,p} = D_0 r( V_{r^*\mathcal{I},0}) + T_p \mathfrak O.
\end{equation}
Moreover, $V_{\mathcal{I},p}$ is invariant under the isotropy representation of $H$.

We are particularly interested in the case that 
$T$ acts on a symplectic manifold $(M,\omega)$ with moment map $\Phi \colon M \to \ft^*$, and
\begin{equation}
    \label{eq:I-Phi}
    \mathcal{I}_{\Phi} = \{ \langle \epsilon, \Phi \rangle \mid  \epsilon \in
C^\infty( M , \ft) \text{ is } T\text{-invariant}\}.
\end{equation}
In this case, to simplify notation we use the shorthand ``modulo $\Phi$" to mean ``modulo $\mathcal{I}_{\Phi}$" and denote the zero locus of the homogenization of $\mathcal{I}_{\Phi}$ at $p$ by $V_{\Phi,p}$.   

    The above discussion holds {\em mutatis mutandis} if we replace the torus $T$ with a compact abelian Lie group.  This is particularly relevant when studying a local model $Y = T \times_H \fh^\circ \times \C^{h+1}$ with moment map $\Phi_Y \colon Y \to \ft^*$. In this case, define 
    \begin{equation*}
    R \colon \C^{h+1} \hookrightarrow T \times_H \fh^\circ \times \C^{h+1} \quad \text{ by }\quad z  \mapsto [1,0,z].
    \end{equation*}
If $\Phi_H \colon \C^{h+1} \to \fh^*$ is the homogeneous moment map for the $H$-action on $\C^{h+1}$, then 
    $$R^*\Phi_Y = \Phi_H \quad \text{and} \quad R^*\mathcal{I}_{\Phi_Y} = \mathcal{I}_{\Phi_H}.$$
Our first important technical lemma concerns  this case.

\mute{Lemma \ref{lemma:hgs-r} is used in \ref{lemma:calculate} (part \ref{item:quot-van}), in \ref{lemma:vanishing_equivalence} (parts \ref{item:taylor-in-ideal} and \ref{item:R-van}), and in \ref{cor:vanishing_set} (parts \ref{item:taylor-in-ideal} and \ref{item:quot-van}).}
\begin{Lemma}\label{lemma:hgs-r}
    Let $Y = T \times_H \fh^\circ \times \C^{h+1}$ be a local model with moment map $\Phi_Y \colon Y \to \ft^*$, and
    let $\Phi_H \colon \C^{h+1} \to \fh^*$ be the homogeneous moment map for the $H$-action on $\C^{h+1}$. If $p := [1,0,0]$ and $\ell \in \mathbb{N}$, then:
   \begin{enumerate}
       \item \label{item:taylor-in-ideal} For any $h \in \mathcal{I}_{\Phi_Y}$, the Taylor polynomial $T_{0}^{\ell} (R^*h)$ vanishes
        on $\Phi_H^{-1}(0)$;
        \item \label{item:quot-van} $ V_{\Phi_Y,p} = D_{0}R(  \Phi^{-1}_H(0)) + T_p \mathfrak O$,       where $\mathfrak O$ is the $T$-orbit through $p$; and
        \item \label{item:R-van} A $T$-invariant function $g \colon Y \to \R$ vanishes  below order $\ell$ at $p$ 
        modulo $\Phi_Y$ exactly if $R^* g$ vanishes below order $\ell$ at $0$ modulo $\Phi_H$.
            \end{enumerate}
\end{Lemma}

\begin{proof}
Define $r \colon \fh^{\circ} \times \C^{h+1} \to Y$  by $r(\alpha,z) = [1,\alpha,z]$ and $\iota \colon \C^{h+1} \to \fh^{\circ} \times \C^{h+1}$ by $\iota(z) = (0,z)$, Recall that $\Phi_Y([t,\alpha,z]) = \alpha + \Phi_H(z)$,  and so $r^*\Phi_Y^\xi: = r^*\langle \Phi_Y, \xi \rangle$ is the $H$-invariant polynomial on $\fh^\circ \times \C^{h+1}$
sending $(\alpha, z)$ to $\langle \alpha + \Phi_H(z), \xi \rangle$ for any $\xi \in \ft$. It is a homogeneous polynomial of degree $2$ for all $\xi \in \fh$ and of degree $1$ for all $\xi \in \fh^\perp$. Hence, there  exist a basis $\{\xi_i\}$ for $\ft$ and non-negative integers $\{d_i\}$ such that $r^* \Phi_Y^{\xi_i}$ is a homogeneous polynomial of degree $d_i$ for each $i$.


Given $h \in \mathcal{I}_{\Phi_Y}$, there exist  $T$-invariant functions
$\epsilon_i \colon Y \to \R$ such that $h = \sum_i \epsilon_i \Phi_Y^{\xi_i}$.
Therefore the Taylor polynomial $T^\ell_{(0,0)} (r^*h)$ is equal to $\sum_i(r^*\Phi_Y^{\xi_i}) \,
T_{(0,0)}^{\ell - d_i} (r^*\epsilon_i) \in r^*\mathcal{I}_{\Phi_Y} $,
where  $T^k_{(0,0)}  (r^*\epsilon_i) \equiv 0$ if $k < 0$.  Since $R = r \circ \iota$,
this implies  that $T^{\ell}_0 ( R^* h) =\iota^*(T^{\ell}_{(0,0)}(r^*h))  \in R^*\mathcal{I}_{\Phi_{Y}}  = \mathcal{I}_{\Phi_H}$, 
and so $T^{\ell}_0 ( R^* h)$ vanishes on $ \Phi_H^{-1}(0)$. This proves \eqref{item:taylor-in-ideal}.

By the preceding paragraph, every polynomial in the homogenization of $r^* \mathcal{I}_{\Phi_Y}$ at $(0,0)$ vanishes on $(r^*\Phi_Y)^{-1}(0) =  \{0\} \times \Phi_H^{-1}(0)$.
Hence,  $\{0\} \times \Phi_H^{-1}(0)$  is contained in  $V_{r^* \mathcal{I}_{\Phi_Y}}$, the zero locus of the homogenization of $r^*\mathcal{I}_{\Phi_Y}$ at $(0,0)$.
Conversely,  since $r^*\Phi_Y^{\xi_i} \in r^*\mathcal{I}_{\Phi_Y} \cap \, Q^{d_i}_{(0,0)}$ and $j^{d_i}_{(0,0)} (r^*\Phi_Y^{\xi_i}) =r^*\Phi_Y^{\xi_i}$, the polynomial $r^*\Phi_Y^{\xi_i}$ is in the homogenization of $r^*\mathcal I_{\Phi_Y}$ at $(0,0)$ for all $i$.  Hence,  
$ V_{r^*\mathcal{I}_{\Phi_Y}}\subseteq\{0\} \times \Phi_H^{-1}(0)$.
 Therefore,  by \eqref{eq:isomorphism}
 $$  V_{\Phi_Y,p} = D_{(0,0)}r(\{0\} \times \Phi^{-1}_H(0)) + T_p \mathfrak O = D_0 R(\Phi^{-1}_H(0)) + T_p \mathfrak O,$$
 and so \eqref{item:quot-van} holds.

 Finally, let $g \colon Y \to \R$ be $T$-invariant.  Assume  that there exists $h \in \mathcal I_{\Phi_Y}$ such that  $R^*(g -h) \in Q^\ell_0$.  Then $0 = T_0^{\ell-1} (R^* (g-h)) = \iota^* T_{(0,0)}^{\ell-1} (r^*(g-h))$,
and so the Taylor polynomial $T_{(0,0)}^{\ell-1} (r^*(g-h))$ lies in the ideal of $H$-invariant functions generated by $\{r^*\Phi^{\xi}_Y \mid \xi \in \fh^{\perp}\}$.  Hence, there exists $f \in \mathcal I_{\Phi_Y}$ such that $r^* f = T_{(0,0)}^{\ell-1} (r^*(g-h))$.
So by replacing $h$ by $h - f$, we may assume that $T_{(0,0)}^\ell (r^* (g - h)) = 0$ 
and so $r^* (g - h) \in Q^{\ell}_0$.   By  \eqref{eq:vanish-r}, this implies that $g - h \in Q^\ell_0$.
Therefore, \eqref{item:R-van} follows  from \eqref{eq:natural}.
\end{proof}

The proofs of parts \eqref{item:taylor-in-ideal} and \eqref{item:quot-van} of Lemma \ref{lemma:hgs-r} rely  heavily on homogeneity of the components of $r^*\Phi_Y$. As shown below, analogous statements fail without this homogeneity. In contrast, the proof of part \eqref{item:R-van} of Lemma  \ref{lemma:hgs-r} only relies on the fact that $r^* \mathcal{I}_{\Phi_Y}$ contains the ideal of $H$-invariant polynomials generated by $\{r^*\Phi^{\xi}_Y \mid \xi \in \fh^{\perp}\}$.

\begin{Example}\label{exm:homogeneous}
Define $\Phi \colon \R^2 \to \R$ by $\Phi(x,y) = x + y^2$. Then $T^1_0 \Phi = x$, which does not vanish on $\Phi^{-1}(0)$. Similarly, $V_{\Phi} = \{0\} \times \R$ is not equal to $\Phi^{-1}(0)$.
\end{Example}

 Next we show that we can determine if a $T$-invariant function on a tall local model $Y$ vanishes below a sufficiently small order modulo $\Phi_Y$ by computing the  reduced Taylor polynomial.

\mute{Lemma \ref{lemma:vanishing_equivalence} is used in Proposition \ref{cor:criterion}.}
\begin{Lemma}\label{lemma:vanishing_equivalence}
Let $Y= T \times_H \fh^{\circ} \times \C^{h+1}$ be a tall local model with moment map $\Phi_Y \colon Y \to \ft^*$  and defining monomial of degree $N$. Fix $0 < \ell \leq N$. A $T$-invariant function $g \colon Y \to \R$ vanishes below order $\ell$ at $p := [1,0,0]$ modulo  $\Phi_Y$  if and only if the reduced Taylor polynomial $\overline{T^{\ell-1}_p g} \colon \YmodT \to \R $ is identically zero.
\end{Lemma}

\begin{proof}
Assume that there exists $h \in \mathcal{I}_{\Phi_Y}$ such that $g- h \in Q^{\ell}_p$. Then  $R^*(g - h) \in Q^{\ell}_0$ by \eqref{eq:natural}, and so $T^{\ell-1}_{0}(R^*g) = T^{\ell-1}_{0}(R^*h)$. By Lemma \ref{lemma:hgs-r}, part \eqref{item:taylor-in-ideal}, the Taylor polynomial $T^{\ell-1}_{0}(R^*h)$  vanishes on $\Phi^{-1}_H(0)$.
Hence, the reduced Taylor polynomial $\overline{T^{\ell-1}_p g} \colon \YmodT \to \R $ is identically zero.

Conversely, assume that reduced Taylor polynomial $\overline{T^{\ell-1}_p g} \colon \YmodT \to \R $ is identically zero.
Since $T^{\ell-1}_0 (R^* g)$ is an $H$-invariant polynomial on $\C^{h+1}$ of degree less than $\ell$,
there exists a unique $T$-invariant polynomial $G$ on $Y$ of degree less than $\ell$
such that $G[t,\alpha,z] = T^{\ell-1}_{0} (R^* g)(z)$; in particular,
 $\overline{G} \colon \YmodT \to R$
is identically zero. We claim that $G \in \mathcal I_{\Phi_Y}$. To see this, recall that by \cite[Lemma 2.14]{ST1} the algebra of $T$-invariant polynomials on $Y$
is generated by the real and imaginary parts of the defining monomial $P \colon Y \to \C$, the map $[t,\alpha,z] \mapsto |z|^2$, and the components of  $\Phi_Y$. Since  the degree of $G$ is less than $N$, this implies that
$G$ lies in the algebra of $T$-invariant polynomials generated by the map $[t,\alpha, z] \mapsto |z|^2$ and the components of the moment map $\Phi_Y$. Hence, we can write $G$ as the sum of an element in the ideal  generated by the components of $\Phi_Y$ and the map $[t,\alpha,z] \mapsto Q(|z|^2)$ for some polynomial $Q$.  Recall that $\Phi_Y([t,\alpha,z]) = \alpha + \Phi_H(z)$, where $\Phi_H \colon \C^{h+1} \to \fh^*$ is the homogeneous moment map for the $H$-action on $\C^{h+1}$. Hence, since $Y$ is tall, there exists $w \in \C^{h+1} \smallsetminus \{0\}$ such that $[1,0,w] \in \Phi_Y^{-1}(0)$. Therefore the element $[1,0,\tau w_0]$ lies in $ \Phi_Y^{-1}(0)$ for all $\tau \in \R$, and so the restriction of the function $[t,\alpha,z] \mapsto |z|^2$ to $\Phi_Y^{-1}(0)$ attains every non-negative number.  Hence, $Q(|z|^2) = 0$ for all $[t,0,z] \in \Phi_Y^{-1}(0)$ exactly if $Q \equiv 0$, and so $G \in \mathcal I_{\Phi_Y}.$

Since $G \in \mathcal I_{\Phi_Y}$, we have that $T^{\ell -1}_{0} (R^*g) = R^* G \in R^*\mathcal{I}_{\Phi_Y} = \mathcal{I}_{\Phi_H}$, and so $R^*g$ vanishes below order $\ell$ at $0$ modulo $\Phi_H$. By Lemma~\ref{lemma:hgs-r}, part \eqref{item:R-van},  this implies that $g$ vanishes below order $\ell$ at $p$ modulo $\Phi_Y$.
\end{proof}

Let $\mathcal{I}$ be a subspace of $C^{\infty}(M)$ and let $g \in C^{\infty}(M)$ vanish below order $\ell$ at $p \in M$ modulo $\mathcal{I}$. The {\bf jet of $\boldsymbol{g}$ of order $\boldsymbol{\ell}$ at $\boldsymbol{p}$ modulo $\bm{\mathcal{I}}$} is the function $j^{\ell}_p(g;\mathcal{I}) \colon V_{\mathcal{I},p} \to \R$ given by 
\begin{equation}
    \label{eq:relative}
    j^\ell_p(g;\mathcal{I})(v):= j^{\ell}_p(g-h)(v) \quad \text{for all } v \in V_{\mathcal{I},p},
\end{equation}
\noindent
where $h \in \mathcal{I}$ and $g - h \in Q^{\ell}_p$. This function does not depend on $h$ because if  $\tilde{h} \in \mathcal{I}$ and $g - \tilde{h} \in Q^{\ell}_p$, then $h - \tilde{h} \in Q^{\ell}_p \cap \mathcal{I}$ and so  $$j^{\ell}_p(g-\tilde{h}) - j^{\ell}_p(g-h) = j^\ell_p(h - \tilde{h})$$
vanishes identically on $V_{\mathcal{I},p}$.

Now assume that $T$ acts on $M$, that $\mathcal{I}$ is a subspace of the algebra of $T$-invariant smooth functions, and that $g$ is $T$-invariant. If $g$ vanishes below order $\ell$ at $p$ modulo $\mathcal{I}$, then  $j^{\ell}_p(g;\mathcal{I})$ is invariant under addition by elements of $T_p \mathfrak O$ and under the slice representation of the stabilizer $H$ of $p$. Hence, the function $j^{\ell}_p(g;\mathcal{I})$ induces a map $ \overline{j^{\ell}_p(g;\mathcal{I})} \colon (V_{\mathcal{I},p}/T_p \mathfrak{O})/ H \to \R$ by setting
\begin{equation}
    \label{eq:reduced_jet}
    \overline{j^{\ell}_p(g;\mathcal{I})}([v + T_p \mathfrak{O}]) := j^{\ell}_p(g;\mathcal{I})(v) \quad \text{for all } v \in V_{\mathcal{I},p}.
\end{equation}

In the context of a $T$-action on a symplectic manifold $(M,\omega)$ with moment map $\Phi \colon M \to \ft^*$, to simplify notation, we denote the jet of $g$ of order $\ell$ at $p$ modulo  $\Phi$ by $j^\ell_p(g;\Phi)$. As we show below, on a local model  we can calculate this jet by computing
the reduced Taylor polynomial.

\mute{Lemma \ref{cor:vanishing_set} is used in Proposition \ref{cor:criterion}.}

\begin{Lemma}\label{cor:vanishing_set}
Let $Y = T \times_H \fh^\circ \times \C^{h+1}$ be a local model with moment map $\Phi_Y \colon Y \to \ft^*$. Let $\mathfrak{O}$ be the $T$-orbit through $p := [1,0,0]$. There is a homeomorphism from
 $(V_{\Phi_Y,p}/T_p \mathfrak O)/H$ to  $Y /\! / T$ that identifies $\overline{j_p^\ell(g; \Phi_Y)}$ and $\overline{T^{\ell}_p g}$
 for every $T$-invariant function $g \colon Y \to \R$ that vanishes below order $\ell \in \N$ at $p$ modulo  $\Phi_Y$.
\end{Lemma}

\begin{proof}
    By Lemma~\ref{lemma:hgs-r}, part \eqref{item:quot-van}, the map $\Phi_H^{-1}(0) /H \to (V_{\Phi_Y,p}/T_p \mathfrak O)/H$ given by $[z] \mapsto [D_{0} R(z) + T_p \mathfrak O]$ is a homeomorphism.
Since the  map $\YmodT \to \Phi^{-1}_H(0)/H$ given by $[t,0,z] \mapsto [z]$ is  also a homeomorphism,  this  implies that  the map $\YmodT \to (V_{\Phi_Y,p}/T_p \mathfrak O)/H$ given by $[t,0,z] \mapsto [D_{0} R(z) + T_p \mathfrak O]$ is a homeomorphism.

If $g \colon Y \to \R$ is a $T$-invariant function that vanishes below order $\ell$ at $p$ modulo  $\Phi_Y$, then there exists $h \in \mathcal{I}_{\Phi_Y}$  such that $g - h \in Q^{\ell}_p$. Hence, for any $[t,0,z] \in \YmodT$,
    \begin{equation*}
        \begin{split}
            \overline{T^{\ell}_p g}([t,0,z]) &= T^{\ell}_{0}(R^*g)(z) = T^{\ell}_{0}(R^*g - R^*h)(z) \\
            &=j^{\ell}_{0}(R^*(g-h))(z) =j^{\ell}_p(g-h)(D_{0}R(z)) \\
            &=\overline{j^{\ell}_p(g;\mathcal{I}_{\Phi_Y})}([D_{0} R(z) + T_p \mathfrak O]),
        \end{split}
    \end{equation*}
    where the first equality follows from the definition of reduced Taylor polynomial; the second from 
     the facts that $z \in \Phi^{-1}_H(0)$ and the Taylor polynomial $T^{\ell}_{0}(R^*h)$ vanishes on $\Phi^{-1}_H(0)$     
    by Lemma \ref{lemma:hgs-r}, part \eqref{item:taylor-in-ideal}; the third from the fact that $R^*(g - h)$ lies in $Q^{\ell}_{0}$ by \eqref{eq:natural}; the fourth from \eqref{eq:natural}; and the last from \eqref{eq:relative} and \eqref{eq:reduced_jet}.
\end{proof}

 By combining Lemmas \ref{lemma:vanishing_equivalence} and \ref{cor:vanishing_set} we obtain the main result of this section: an intrinsic characterization of ephemeral points.

\mute{Proposition \ref{cor:criterion} is used in Proposition \ref{prop:example}.}
\begin{Proposition}\label{cor:criterion}
Let $\left(M,\omega, \ft \times \R, f:=(\Phi,g)\right)$ be an integrable system such that $(M,\omega,\Phi)$ is a complexity one $T$-space. Assume that $p \in \Phi^{-1}(0) \cap g^{-1}(0)$ is tall and that the  degree $N$ of the defining monomial of the local model for $p$ is greater than one. Then the following are equivalent:
\begin{enumerate}[label=(\arabic*),ref=(\arabic*)]
    \item \label{item:ephemeral} $p$ is ephemeral, i.e., conditions  (a) and (b) of Definition \ref{eph} hold for {\it  some} a $T$-equivariant symplectomorphism from a $T$-invariant neighborhood of $p$ to an open subset of $Y$ that takes $p$ to $[1,0,0]$,
    \item \label{item:forall} conditions (a) and (b) of Definition \ref{eph} hold for {\it every} $T$-equivariant symplectomorphism from $T$-invariant neighborhood of $p$ to an open subset of $Y$ that takes $p$ to $[1,0,0]$, and
    \item \label{item:intrinsic} $g$ vanishes below order $N$ at $p$ modulo $\Phi$  and the zero set of $\overline{j_p^N(g;\Phi)} \colon (V_{\Phi}/T_p \mathfrak{O})/H \to \R$ is homeomorphic to $\R$.
\end{enumerate}
\end{Proposition}

\begin{proof} By the Marle-Guillemin-Sternberg local normal form theorem, \ref{item:ephemeral} is an immediate consequence of \ref{item:forall}.
By the naturality of our constructions (see \eqref{eq:natural} and \eqref{eq:natural2}),
the implications \ref{item:ephemeral} $\Rightarrow$ \ref{item:intrinsic} and \ref{item:intrinsic} $\Rightarrow$ \ref{item:forall} follow immediately from Lemmas \ref{lemma:vanishing_equivalence} and \ref{cor:vanishing_set}. 
\end{proof}

 Finally, we show how to explicitly check whether the intrinsic criteria
described in Proposition~\ref{cor:criterion}, part \ref{item:intrinsic} are satisfied.
As we see in Section~\ref{sec:examples}, this allows us to determine in practice which points are ephemeral.

\mute{Lemma~\ref{lemma:calculate} is used in Section 7.}
\begin{Lemma}\label{lemma:calculate}
    Let $(M,\omega,\Phi)$ be a complexity one $T$-space. Fix $p \in \Phi^{-1}(0)$.
    \begin{enumerate}[label=(\Alph*),ref=\Alph*]
        
    \item \label{item:quotient} If $\mathfrak{O}$ is the $T$-orbit through $p$ and $H$ is the stabilizer of $p$, then
 $$ (V_{\Phi,p}/T_p \mathfrak{O})/H = \hat{\Phi}^{-1}(0)/H ,$$
 where $\hat{\Phi} \colon (T_p \mathfrak O)^{\omega}/T_p\mathfrak O \to \fh^*$ is the homogeneous
 moment map for the symplectic slice representation. 
 \item \label{item:relative-jet} Let $g \colon M \to \R$ be a $T$-invariant function that vanishes below order $\ell \in \N$ at $p$ modulo $\Phi$. Then for any $h \in \mathcal{I}_{\Phi}$ such that $g - h \in Q^{\ell}_p$ and for any coordinate chart $\varphi$ centered at $p$
 $$ \overline{j_p^\ell(g; \mathcal{I})}([v + T_p \mathfrak{O}]) = T_{\varphi(p)}^\ell (g \circ \varphi^{-1}\! -\! h \circ \varphi^{-1})(D_p \varphi(v))$$
 for all $v \in V_{\Phi,p}$.
 \end{enumerate}
\end{Lemma}

\begin{proof}
Since $V_{\Phi,p}$ and $\hat \Phi$ are both natural with respect to $T$-equivariant
symplectomorphisms that intertwine moment maps, by the Marle-Guillemin-Sternberg local normal form theorem
we may assume that $M$ is a local model 
$Y = T \times_H \mathfrak h^\circ \times \C^{h+1}$, $\Phi = \Phi_Y$, and $p = [1,0,0]$.
We may identify $T_p Y$ with $\fh^{\perp} \times \fh^{\circ} \times \C^{h+1}$, equipped with the product of the standard symplectic forms on  $\fh^{\perp} \times \fh^{\circ} \subseteq \ft \times \ft^*$ and $\C^{h+1}$. Since $T_p \mathfrak O  = \fh^{\perp} \times \{0\} \times \{0\}$, this implies that $(T_p \mathfrak O)^{\omega} = \fh^{\perp} \times \{0\} \times \C^{h+1}$. Hence,  $D_0 R$ induces an $H$-equivariant symplectomorphism from $T_0 \C^{h+1} \simeq \C^{h+1}$ to the symplectic slice 
$(T_p \mathfrak O)^{\omega}/T_p\mathfrak O$. Therefore, $\hat \Phi( D_0 R(z)) = \Phi_H(z),$ and so
 $$ \hat \Phi^{-1}(0) = D_{0} R(\Phi_H^{-1}(0)) + T_p \mathfrak O.$$ 
 Part \eqref{item:quotient} now follows from Lemma~\ref{lemma:hgs-r}, part \eqref{item:quot-van}.  Part \eqref{item:relative-jet} is an immediate consequence of \eqref{eqn:defn_jet}, \eqref{eq:relative} and \eqref{eq:reduced_jet}.
\end{proof}

\section{A family of examples}\label{sec:examples}

\mute{In this section we use: Lemma~\ref{lemma:defn_poly}, Lemma \ref{trivial},  Lemma \ref{greg2}, Lemma \ref{lemma:calculate}, Proposition \ref{cor:criterion}.}
In this section we explicitly construct a natural family of integrable
systems that satisfy the criteria of the main result of this paper, Theorem \ref{thm:main}, but do not satisfy the criteria of \cite[Theorem 1.6]{ST1} because they have tall degenerate singular points. This shows that allowing ephemeral singular points meaningfully
extends the integrable systems that we can study.

Let the $(n-1)$-dimensional torus $T$ act on $\C^n$ via an injective homomorphism $\rho \colon T \to (S^1)^n$ with  proper moment map $\Phi \colon \C^n \to \ft^*$.
By Lemma~\ref{lemma:defn_poly}, there exists $\xi \in \Z^n$ such that $\lambda \in \rho(T) \varsubsetneq (S^1)^n$ exactly if $\prod_i \lambda_i^{\xi_i} = 1$. Define a function $g \colon \C^n \to \R$  by
$$g(z) =  \Im \left( \, \prod_{\xi_i > 0} z_i^{\xi_i} \prod_{\xi_i < 0} \overline{z}_i^{\, -\xi_i} \right), $$
where the first (respectively second) product is over all $i$ such that $\xi_i >0$ (respectively $\xi_i <0$).

\mute{03/05/24: we believe that the last sentence in Proposition \ref{prop:example} can be turned into an iff.}

\begin{Proposition}\label{prop:example}
The quadruple $(\C^n, \omega_{\C^n}, \ft \times \R, f:=(\Phi, g))$ is an integrable system such that $(\C^n, \omega_{\C^n}, \Phi)$ is a proper complexity one $T$-space.  Moreover, each tall singular point in $(\C^n, \omega_{\C^n}, \ft \times \R, f:=(\Phi, g))$ either has purely elliptic type or is ephemeral. Finally, the system has  degenerate ephemeral singular points if and only if $|\sum_{i \in I} \xi_i | > 2$ for some $I \subseteq \{1,\dots,n\}$.
\end{Proposition}

\begin{proof}
Since $g$ is $T$-invariant by construction, and $\Phi$ is a proper moment map for the $T$-action, $f$ is the moment map of a $(T \times \R)$-action.

Consider a point $w = (w_1,\dots,w_n) \in \C^n$. Write  $w_j = (r_j, \theta_j)$ in polar coordinates for all $j$, and 
let \begin{equation}
   \label{eq:I}
    I := \{i \in \{1,\ldots,n\} \mid w_i = 0\}.
\end{equation}
We begin by summarizing our claims about $w$, which we prove below: 
First, if $I = \emptyset$, then $w$ is a singular point of $f$ exactly if \eqref{cond 1} and \eqref{cond 2} hold;
in particular, $f$ is regular on an open dense set in $\C^n$.
More generally,
$w$ is tall exactly if $\xi_i \xi_j \geq 0$ for all $i, j \in I$;
assume that this is the case. 
If $w$ is singular and $\sum_{i \in I } |\xi_i| \leq  1$, then $w$ has purely elliptic type.
In contrast,
if $\sum_{i \in I} |\xi_i| > 1$ then $w$ is  an ephemeral singular point;
while if $\sum_{i \in I} |\xi_i| > 2$, then $w$ is a degenerate point of $f$. Together these  claims imply the result.

Fix $w \in \C^n$, and set
\begin{equation}
\label{eq:spaces-I}
    \begin{split}
        (S^1)^{I} &:= \{ \lambda \in (S^1)^n \mid \lambda_j = 1  \text{ for all } j \notin I \} \quad \text{and} \\
        \C^{I} & := \{z \in \C^n \mid z_j = 0 \text{ for all } j \notin I\},
    \end{split}
\end{equation}
where $I$ is defined as in \eqref{eq:I}.  
 To prove the above claims, we split the proof in two cases, depending on whether $\sum_{i \in I} |\xi_i|$ vanishes or not.

{\bf Case 1:  $\sum_{i \in I} |\xi_i| = 0$.} Since $\xi_i = 0$ if $i \in I$, we may write
\begin{equation}
\label{eq:polar}
    \Phi(w) = \textstyle \frac{1}{2} \sum_j \eta_j r_j^2 \quad \mbox{and} \quad g(w) = \bigg( \prod\limits_{j \notin I} r_j^{|\xi_j|} \bigg) \sin \bigg( \sum\limits_{j \notin I} \xi_j \theta_j\bigg),
\end{equation}
where $\eta_1,\dots, \eta_n  \in \ft^*$ are the isotropy weights for the $T$-action. Hence,

\begin{subequations}
    \begin{align}
        \label{eq:der1} D_w \Phi &= \textstyle \sum_j \eta_j r_j dr_j,  \\
      \label{eq:span} \ker \, D_w \Phi  &= \C^I \oplus \operatorname{span}\bigg( \big\{\partial_{\theta_j} \mid j \notin I \big\}  \cup 
    \bigg\{  \textstyle \sum\limits_{j \notin I} \frac{\xi_j}{r_j} \partial_{r_j} \bigg\}  \bigg),\text{ and}\\
        d_w g &= \textstyle g(w) \bigg( \! \sum\limits_{j \notin I} \frac{|\xi_j|}{r_j} d r_j\bigg) + \bigg(\prod\limits_{j \notin I} r_j^{|\xi_j|} \bigg)\! \cos \!  \bigg(\!\sum\limits_{j \notin I} \xi_j \theta_j \! \bigg) \! \! \sum\limits_{j \notin I} \xi_j  d\theta_j. \label{eq:der2}
    \end{align}
\end{subequations}
Here, we identify $\C^I$ with the span of $\{ \del_{x_i}, \del_{y_i} \mid i \in I\}$.
Since  $ \xi_j \neq 0$  for some $j \notin I$, the differential $d_w g$ vanishes
on the span of  $\{\partial_{\theta_j} \mid j \not \in I \}$ exactly if
\begin{equation}
\label{cond 1}
    \textstyle \cos \bigg(\sum\limits_{j \notin I} \xi_j\theta_j\bigg) = 0.
\end{equation}

Moreover, if equation \eqref{cond 1} holds, then  $g(w) \neq 0$.  Hence, $w$ is a critical point of $g$ modulo $\Phi$, that is,  the derivative $d_w g$ vanishes on $\ker D_w \Phi$, if and only if \eqref{cond 1} holds and
\begin{equation}
    \label{cond 2}
    \sum\limits_{j \notin I} \frac{\xi_j |\xi_j|}{r_j^2} = 0.
\end{equation}
Furthermore,  if $I = \emptyset$  then  $D_w\Phi$  has rank $n-1$, and so  $w$ is singular for $f$  exactly if $w$ is a critical point of $g$ modulo $\Phi$,
that is, exactly if  \eqref{cond 1} and \eqref{cond 2} hold.

More generally, let $H$ be the $T$-stabilizer of $w$ and let $\mathfrak O$ be the $T$-orbit through $w$. Since $\sum_{i \in I} |\xi_i| = 0$, we have $H = (S^1)^I$ and 
  $$ T_w \mathfrak O = \big\{ \sum_{j \notin I} a_j \del_{\theta_j} \ \big| \ \sum_{j \notin I} \xi_j a_j = 0 \big\}.$$
Since $(T_w \mathfrak O)^{\omega} = \ker D_w \Phi$,
\eqref{eq:span} implies that,
 as symplectic vector spaces,
\begin{equation}
    \label{eq:slice}
    (T_w \mathfrak O)^{\omega}/T_w \mathfrak O \simeq \C^I \oplus \operatorname{span} \bigg(\sum_{j \notin I}  \xi_j \del_{\theta_j}, \sum_{j \notin I} \frac{\xi_j}{r_j} \partial_{r_j}\bigg) \simeq \C^I \oplus \R^2.
\end{equation}
Under this identification, $H$ acts on $\C^I$ by an isomorphism with $(S^1)^I$
and trivially on $\R^2$. In particular, $w$ is tall
by Lemma~\ref{lemma:defn_poly}.

Now assume that $w$ is a singular point for $f$. By Lemma \ref{greg2}, if $w$ is regular for $g$ modulo $\Phi$, then $w$ has purely elliptic type. Hence  we may assume that $w$ is a critical point of $g$ modulo $\Phi$. Since $d_w g$ vanishes on $\ker D_w \Phi$, there exists $\mu \in \ft$ such that $\tilde{g}:= g - \Phi^{\mu}$ has vanishing derivative at $w$. By \eqref{eq:der1}, 
$$\textstyle d_w \Phi^{\mu} = \sum_j \langle \eta_j ,\mu \rangle r_j dr_j = \sum_j \frac{\partial^2 \Phi^{\mu}}{\partial r^2_j}(w) r_j dr_j. $$ 
Since $d_w g = d_w \Phi^{\mu}$, by \eqref{eq:der2} and \eqref{cond 1}  this implies that $\textstyle \frac{\partial^2 \Phi^{\mu}}{\partial r_j^2} (w)= \textstyle \frac{|\xi_j|}{r^2_j} g(w)$ for all  $j \notin I$. Moreover, all other components of the Hessian $d^2_w \Phi^{\mu}$ vanish. Hence, by \eqref{eq:der2} and  \eqref{cond 1}, the only possibly non-zero components of the Hessian $d^2_w \tilde{g} = d^2_w g - d^2_w \Phi^{\mu}$ are given by 
\begin{equation}
\label{eq:hessian}
    \begin{split}
    \frac{\partial^2 \tilde{g}}{\partial \theta_j \partial \theta_k}(w) & = \textstyle -  \xi_j \xi_k g(w)  , \\
       \frac{\partial^2 \tilde{g}}{\partial r_j \partial r_k}(w)  & = \frac{|\xi_j| |\xi_k| }{r_j r_k} g(w)  - 2 \delta_{jk} \frac{|\xi_j|}{r^2_j}g(w),
    \end{split}
\end{equation}
where $\delta_{jk}$ is the Kronecker delta and $j,k \notin I$.

Since $w$ is a critical point of $g$ modulo $\Phi$, the $(T \times \R)$-orbit $\mathcal{O}$ through $w$ is equal to the $T$-orbit $\mathfrak{O}$, and so \eqref{eq:slice} also describes $(T_w \mathcal O)^{\omega}/T_w \mathcal O$ as a symplectic vector space. We may replace $f = (\Phi,g)$ by $\tilde{f}:=(\Phi,\tilde{g})$ because this does not change our analysis. Then the $(T \times \R)$-stabilizer of $w$ is $H \times \R$. Moreover, since $H$ acts on $\C^I$ by an isomorphism with $(S^1)^I$ with homogeneous moment map $\Phi_H \colon \C^I \to \fh^*$ and acts trivially on $\R^2$, the homogeneous moment map for the linearized action of $H \times \R$ at $w$ can be identified with the restriction of $(\Phi_H, d^2_w \tilde{g})$ to $\C^I \oplus \R^2$.  Finally, since $\Phi_H$ (respectively $d^2_w \tilde{g}$) only depends on coordinates of the first (respectively second) summand, and since $H$ acts in a toric fashion on $\C^I$, it suffices to consider the restriction of $d^2_w \tilde{g}$ to $\operatorname{span} (\sum_{j \notin I}  \xi_j \del_{\theta_j}, \sum_{j \notin I} \frac{\xi_j}{r_j} \partial_{r_j}) \simeq  \R^2$.

By \eqref{cond 2} and \eqref{eq:hessian}, this Hessian is the $2 \times 2$ diagonal matrix $A$  with non-zero entries  

\begin{equation*}
    \begin{split}
        d^2_w \tilde{g}\bigg(\sum_{j \notin I} \xi_j \partial_{\theta_j}, \sum_{k \notin I} \xi_k \partial_{\theta_k}  \bigg) & = - \bigg(\sum_{j \notin I} \xi^2_j\bigg)^2 g(w), \\
        d^2_w \tilde{g} \bigg(\sum_{j \notin I}\frac{\xi_j}{r_j} \partial_{r_j}, \sum_{k \notin I} \frac{\xi_k}{r_k} \partial_{r_k} \bigg) & = \bigg(\sum_{j \notin I} \frac{|\xi_j|\xi_j}{r_j^2}\Big)^2 g(w) -2 \bigg( \sum_{j \notin I} \frac{\xi_j^2|\xi_j|}{r_j^4}\bigg)g(w) \\
        & = -2 \bigg( \sum_{j \notin I} \frac{\xi_j^2|\xi_j|}{r_j^4}\bigg)g(w).
    \end{split}
\end{equation*}
Since $\xi \neq 0$ and since $g(w) \neq 0$ by \eqref{cond 1}, these diagonal entries are  non-zero and have the same sign.  Hence, if $J = \big(\begin{smallmatrix} 0 & -1 \\
1 & 0\end{smallmatrix}\big)$, then the product $JA$ has distinct non-zero imaginary eigenvalues. Hence,  $w$ has purely elliptic type; see, e.g., \cite[Remark 6.1]{ST1}.

{\bf Case 2:  $N:=\sum_{i \in I} |\xi_i| >  0$.} 
Let $H$ be the $T$-stabilizer of $w$ and let $\mathfrak O$ be the $T$-orbit through $w$.
Since there exists $i \in I$ with $\xi_i \neq 0$,
the map $(S^1)^{I} \to S^1$ taking $\lambda$ to $\prod_{i \in I} \lambda_i^{\xi_i}$ is surjective. Hence, $\dim T/H =
(n-1) - (|I|- 1) = n - |I|$, and so the tangent space $T_w \mathfrak O$ is the span of $\{\frac{\partial}{\partial \theta_j} \mid j \notin I\}$.  Then, as symplectic vector spaces
\begin{equation}
            \label{eq:slice2}
    (T_w \mathfrak O)^\omega/T_w \mathfrak O \simeq \C^{I},
\end{equation}
where $\C^I$ is defined as in \eqref{eq:spaces-I}. Moreover, $H$
acts on $\C^I$ through an
isomorphism with the subgroup 
$$\bigg\{\lambda \in (S^1)^I \, \bigg|  \, \prod_{i \in I} \lambda_i^{\xi_i} = 1 \bigg\}.$$
Hence, by  Lemma \ref{lemma:defn_poly}, $w$  is  tall
exactly if $\xi_i \xi_j \geq 0$ for all $i,j \in I$.

Since we are interested only in tall singular points,  we may assume that $\xi_i \geq 0$ for all $i \in I$; the argument when $\xi_i \leq 0$ for all $i$ is entirely analogous.
In this case, the defining monomial of the local model for $w$ is given by 
$P(z) = \prod_{i \in I} z_i^{\xi_i}$, and  has degree $N$. Moreover, every partial derivative of $g$ of degree less than $N$ vanishes at $w$. Moreover,  $\prod_{j \notin I} w_j^{\xi^+_j} \overline w_j^{\xi^-_j} \neq 0$, where $\xi^+_j = \operatorname{max} \{0, \xi_j\}$ and $\xi^-_j = \operatorname{max} \{0, -\xi_j\}$ for any $j \notin I$. Hence, there are precisely two partial derivatives of $g$ with respect to $z_1, \dots, z_n$ and $\overline z_1, \dots, \overline z_n$ of degree $N$ that do not vanish at $w$: 
\begin{equation*}
    \begin{split}
        \frac{ \partial^N g}{\partial^{\xi_{i_1}} z_{i_1} \ldots \partial^{\xi_{i_k}}z_{i_k}}(w) &=  -\frac{\sqrt{-1}}{2} \prod\limits_{i \in I }\xi_{i}! \prod\limits_{j \notin I } w_j^{\xi_j^+} \overline w_j^{\xi_j^-}, \\
        \frac{ \partial^N g}{\partial^{\xi_{i_1}} \overline{z}_{i_1} \ldots \partial^{\xi_{i_k}}\overline{z}_{i_k}}(w) &=  \frac{\sqrt{-1}}{2} \prod\limits_{i \in I}\xi_{i}! \prod\limits_{j \notin I } \overline w_j^{\xi_j^+} w_j^{\xi_j^-}, 
    \end{split}
\end{equation*}
where $I = \{i_1,\ldots, i_k\}$. Therefore the  degree $N$ Taylor polynomial of $g$ at $w$ is given by
   \begin{equation}
   \label{eq:taylorpoly}
       \textstyle T^N_w g(z) =\Im \left( \prod\limits_{j \notin I} w_j^{\xi^+_j} \overline w_j^{\xi^-_j}P(z) \right)\quad \forall \, z \in \C^n.
   \end{equation}

By \eqref{eq:slice2}, we may identify 
the symplectic slice $(T_w \mathfrak O)^\omega/T_w \mathfrak O$ with $\C^I$ so that the homogeneous moment map for the symplectic slice representation $\hat{\Phi} \colon (T_w \mathfrak O)^\omega/T_w \mathfrak O \to \fh^*$ becomes the homogeneous moment map $\Phi_H \colon \C^I \to \fh^*$ for the $H$-action on $\C^I$. 
If $N = 1$, then the restriction of \eqref{eq:taylorpoly} to the symplectic slice $\C^{I}$ is a non-trivial linear map. Hence,  $w$
is a regular point of $g$ modulo $\Phi$ and so, by  Lemma \ref{greg2}, $w$ has purely elliptic type. Therefore,  we may assume that $N >1$, and so $w$ is a singular point by Lemma \ref{greg2}. 
By \eqref{eq:taylorpoly} the function $g$ vanishes below order $N$ at $w$. Hence, replace $\Phi$ by $\Phi - \Phi(w)$, and then
apply Lemma \ref{lemma:calculate}, part \eqref{item:relative-jet} with the identity as a coordinate chart, and $h \equiv 0$; by equation \eqref{eq:taylorpoly}, the jet of $g$ of order $N$ at $w$ modulo $\Phi$
induces
the map $\overline{j^N_w(g;\Phi)} \colon (V_{\Phi,w}/T_w \mathfrak{O})/ H \to \R$ given by
\begin{equation}
    \label{eq:relative_jet}
    \textstyle 
\overline{j^N_w(g;\Phi)} ([z]) =  \Im \bigg( \prod\limits_{j \notin I} w_j^{\xi^+_j} \overline w_j^{\xi^-_j} P(z) \bigg) \quad \forall \, [z] \in  (V_{\Phi,w}/T_w \mathfrak{O})/ H.
\end{equation}
 
By Lemma \ref{lemma:calculate}, part \eqref{item:quotient}, we have that $(V_{\Phi,w}/T_w \mathfrak{O})/ H = \Phi^{-1}_H(0)/H$; moreover, by Lemma \ref{trivial} the map $\Phi^{-1}_{H}(0)/H \to \C$ given by $[z] \mapsto P(z)$ is a homeomorphism.  Using this map to identify $(V_{\Phi,w}/T_w \mathfrak{O})/ H$ with $\C$, we may rewrite \eqref{eq:relative_jet} as
$$\overline{j^N_w(g;\Phi)} (v) =  \Im \bigg( \prod\limits_{j \notin I} w_j^{\xi^+_j} \overline w_j^{\xi^-_j}v \bigg) \quad \forall \ v \in \C.$$
 Hence, the zero locus of $\overline{j^N_w(g;\Phi)}$ is homeomorphic to $\R$ and so, by Proposition \ref{cor:criterion}, $w$ is  ephemeral.

   Finally, if $N > 2$, then by \eqref{eq:taylorpoly} the degree $2$ Taylor polynomial of $g$  at $w$ vanishes on $\C^n$.
   Hence $w$ is a critical point of $g$ and so the $(T \times \R)$-stabilizer of $w$ is $H \times \R$.
Moreover, if $\mathcal{O}$ denotes the $(T \times \R)$-orbit through $w$, then the homogeneous moment map for the linearized action of $\R$ on $(T_w \mathcal{O})^{\omega}/T_w \mathcal O$ 
is trivial. Hence, $w$ is a degenerate singular point; see, e.g., \cite[Remark 6.1]{ST1}.  
\end{proof}


\begin{thebibliography}{1}

\bibitem[At]{At} M. F. Atiyah, {\bf Convexity and Commuting Hamiltonians}, Bull. London Math. Soc. 14 (1982), no. 1, 1 -- 15.

\bibitem[CB]{CB} R. Castaño Bernard, {\bf Symplectic invariants of some families of Lagrangian $T^3$-fibrations}, J. Symplectic Geom. 2 (2004), no. 3, 279 -- 308. 

\bibitem[DK]{DK} J. J. Duistermaat and J. A. C. Kolk, {\bf Lie Groups}, Universitext, Springer, Berlin, (2000).

\bibitem[E]{E} K. Efstathiou, {\bf Metamorphoses of Hamiltonian Systems with Symmetries}, Lecture Notes in Mathematics 1864, Springer, Berlin, (2005).



\bibitem[GS]{GS} V. Guillemin and S. Sternberg, {\bf A normal form for the moment map}, Differential geometric methods in mathematical physics
  ({J}erusalem, 1982), Math. Phys. Stud., 6, Reidel, Dordrecht,
  1984, 161 -- 175.

\bibitem[GS2]{GS2} V. Guillemin and S. Sternberg, {\bf Convexity properties of the moment mapping},  Invent. Math. 67 (1982), no. 3, 491 -- 513.

\bibitem[HL]{HL} R. Harvey and H. B. Lawson Jr., {\bf Calibrated Geometries}, Acta Math. 148 (1982), 47 -- 157.
  
\bibitem[HSSS]{HSSS} S. Hohloch, S. Sabatini, D. Sepe and M. Symington, {\bf From Hamiltonian $S^1$-spaces to compact semi-toric systems}, in preparation.

\bibitem[HSS]{HSS} S. Hohloch, S. Sabatini and D. Sepe,  {\bf From compact semi-toric systems to Hamiltonian $S^1$-spaces}, Discrete Contin. Dyn. Syst. 35 (2015), no. 1, 247 -- 281.


\bibitem[J]{J} D. Joyce, {\bf Singularities of special Lagrangian fibrations and the SYZ conjecture}, Commun. Anal. Geom. 11 (2003), no. 5, 859 -- 907.

\bibitem[KT1]{KT1} Y. Karshon and S. Tolman, {\bf Centered complexity one {H}amiltonian torus actions}, Trans. Amer. Math. Soc. 353 (2001), no. 12, 4831 -- 4861.

\bibitem[LMTW]{LMTW} E. Lerman, E. Meinrenken, S. Tolman and C. Woodward, {\bf Non-abelian convexity by symplectic cuts}, Topology 37 (1998), no. 2, 245 -- 259.

\bibitem[Li]{Li} H. Li, {\bf The fundamental group of symplectic manifolds with Hamiltonian Lie group
actions}, J. Symplectic Geom. 4 (2006), no. 3, 345 -- 372.

\bibitem[M]{M} C.-M. Marle, {\bf Mod\`ele d'action hamiltonienne d'un groupe de {L}ie sur une
    vari\'{e}t\'{e} symplectique}, Rend. Sem. Mat. Univ. Politec. Torino 43 (1985), no. 2, 227 -- 251.

\bibitem[Mil]{Mil} J. Milnor, {\bf Morse Theory}, Princeton University Press, Princeton, 1961.

\bibitem[NSZ]{NSZ} N. N. Nekhoroshev, D. A. Sadovskií, and B. I. Zhilinskií, {\bf Fractional Hamiltonian Monodromy},
Ann. Henri Poincaré 7 (2006), no. 6, 1099 -- 1211.

\bibitem[SD]{SD} S. Schmidt and H. R. Dullin, {\bf Dynamics near the $p:-q$ resonance}, Physica D 239 (2010), no. 19, 1884 -- 1891.

\bibitem[Sch]{Sch} G. Schwarz, {\bf Smooth functions invariant under the action of a compact Lie group}, Topology 14 (1975), 63 -- 68.

\bibitem[Sja]{Sja} R. Sjamaar, {\bf Convexity properties of the moment mapping re-examined}, Adv. Math. 138 (1998), no. 1, 46 -- 91.


\bibitem[ST1]{ST1} D. Sepe and S. Tolman, {\bf Connectedness of fibers beyond semitoric systems I: the non-degenerate case}, preprint, arXiv:2402.05814.

\bibitem[VN]{VN} S. V\~u Ngoc, {\bf Moment polytopes for symplectic manifolds with monodromy}, Adv. Math. 208 (2007), no. 2, 909 -- 934.

  
\bibitem[Wa]{Wa} C. Wacheux, {\bf Syst\`emes int\'egrables semi-toriques et
polytopes moment}, Ph.D. Thesis, Universit\'e de Rennes 1, 2013.
  

\end{thebibliography}
\end{document}